\documentclass[a4paper,twoside]{article}  

\usepackage{amssymb,amsmath,amsthm,dsfont,amsfonts,color,latexsym,CJK}

\usepackage{hyperref}

\usepackage{mathrsfs} 

\definecolor{blue}{rgb}{0.00,0.00,1.00}
\definecolor{red}{rgb}{1.00,0.00,0.00}

\renewcommand{\baselinestretch}{1.2}
\def\bq{\begin{equation}}
\def\eq{\end{equation}}
\def\ba{\begin{array}{ccc}}
\def\bal{\begin{array}{lll}}
\def\ea{\end{array}}

 \def\lt#1{\left#1}\def\rt#1{\right#1}
\def\({\left(}\def\){\right)}
\def\[{\left[}\def\]{\right]}
\def\<{\langle}\def\>{\rangle}

    \def \O   {\mathbb{O}}

    \def \R   {\mathbb{R}}

    \def\eps  {\epsilon}
    \def\intr {\int_{\R^3}}
    \def\intrr {\int_{\R^6}}
    
    \def\intt {\int^t_0}

    \def \pt   {\partial}
    
    \def \Dt   {\frac{\rm d}{{\rm d}t}}
    
    \def \dt    {\partial_t}

    \def \dxa   {\partial^{\alpha}_x}
    \def \dxb   {\partial^{\beta}_x}

    \def \dvb   {\partial^{\beta}_v}

    \def \divx  {{\rm div}_x}

    \def\Tdx   {\nabla_x}
    \def\Tdv   {\nabla_v}

       \def\bq{\begin{equation}}
       \def\eq{\end{equation}}
       \def\be{\begin{equation}}
       \def\ee{\end{equation}}
       \def\bma#1\ema{{\allowdisplaybreaks\begin{align}#1\end{align}}}
       \def\bmas#1\emas{{\allowdisplaybreaks\begin{align*}#1\end{align*}}}
       \def\bln#1\eln{{\allowdisplaybreaks\begin{aligned}#1\end{aligned}}}
       \def\nnm{\notag}
       \def\bgr#1\egr{\allowdisplaybreaks\begin{gather}#1\end{gather}}
       \def\bgrs#1\egrs{\allowdisplaybreaks\begin{gather*}#1\end{gather*}}

       \theoremstyle{plain}
       \newtheorem{lem}{\bf Lemma}[section]
       \newtheorem{thm}[lem]{\textbf{Theorem}}

\begin{document}

\title{ Diffusion Limit and the optimal convergence rate of the Vlasov-Poisson-Fokker-Planck system }
\author{   Mingying Zhong$^*$\\[2mm]
 \emph{\small\it  $^*$College of  Mathematics and Information Sciences,
    Guangxi University, P.R.China.}\\
    {\small\it E-mail:\ zhongmingying@sina.com}\\[5mm]
    }
\date{ }

\pagestyle{myheadings}
\markboth{Vlasov-Poisson-Fokker-Planck system }%
{ M.-Y. Zhong }

 \maketitle

 \thispagestyle{empty}

\begin{abstract}\noindent
In the present paper, we study the diffusion limit of the classical solution to the Vlasov-Poisson-Fokker-Planck (VPFP) system with initial data near a global Maxwellian. We prove the convergence and establish the optimal convergence rate of the global strong solution to the VPFP system towards the solution to the  drift-diffusion-Poisson  system based on the spectral analysis with precise estimation on the initial layer.

\medskip
 {\bf Key words}.  Vlasov-Poisson-Fokker-Planck system,  spectral analysis, diffusion limit, convergence rate.

\medskip
 {\bf 2010 Mathematics Subject Classification}. 76P05, 82C40, 82D05.
\end{abstract}

%

\tableofcontents

\section{Introduction}
The Vlasov-Poisson-Fokker-Planck  (VPFP)  system can be used to model the time evolution of dilute charged particles governed by the electrostatic force coming from their (self-consistent)
Coulomb interaction. The collision term in the kinetic equation is the Fokker-Planck operator that describes the Brownian force.  In general, the rescaled VPFP system defined on $\R^3_x\times\R^3_v$ takes the form
 \bgr
\dt F_{\eps}+\frac{1}{\eps}v\cdot\Tdx F _{\eps}+\frac{1}{\eps}\Tdx \Phi_{\eps}\cdot\Tdv F_{\eps}=\frac{1}{\eps^2}\Tdv\cdot(\Tdv F_{\eps}+vF_{\eps}),\label{VPFP1}\\
\Delta_x\Phi_{\eps}=\intr F_{\eps} dv-1,\label{VPFP3}\\
 F_{\eps}(0,x,v)=F_{0}(x,v),  \label{VPFP3i}
\egr
where $\eps>0$ is a small parameter related to the mean free path, $F_{\eps}=F_{\eps}(t,x,v)$ is the number
density function of charged particles, and   $\Phi_{\eps}(t,x)$ denotes the electric potential, respectively.
Throughout this paper, we assume $\eps\in(0,1)$.

In this paper, we study the diffusion limit
of the strong solution to the rescaled VPFP system~\eqref{VPFP1}--\eqref{VPFP3i}  with initial data near the normalized Maxwellian $M(v)$ given by
$$
 M=M(v)=\frac1{(2\pi)^{3/2}}e^{-\frac{|v|^2}2},\quad v\in\R^3.
$$

Our motivation  is to prove the
convergence and establish the convergence rate of strong solutions $(F_{\eps}, \Phi_{\eps})$ to \eqref{VPFP1}--\eqref{VPFP3i} towards  $(N M, \Phi)$, where $(N,\Phi)(t,x)$ is the solution of the following drift-diffusion-Poisson (DDP) system:
\bgr
\dt N+\Tdx\cdot(\Tdx N+N\Tdx \Phi)=0, \label{DDP1}\\
\Delta_x\Phi=N-1,\label{DDP2}\\
N(0,x)=N_0(x)=\intr F_0 dv.\label{DDP3}
\egr
Here, $N=N(t,x)$ stands for densities of charged particles in the ionic
solution and $\Phi(t,x)$ is the self-consistent electric potential.

The DDP system \eqref{DDP1}--\eqref{DDP3} provides a continuum description of the evolution of
charged particles via macroscopic  quantities, for example, the particle
density, the current density etc., which have cheaper costs for numerics. Such
continuum models can be (formally) derived from kinetic models by coarse graining
methods, like the moment method, the Hilbert expansion method and so on \cite{DDP-5,Markowich,Maxwell}.  Concerning the mathematical analysis, the initial value problem and the initial
boundary value problem of the DDP system have been extensively studied, we refer
to \cite{DDP-0,DDP-1,DDP-2,DDP-3,DDP-4,DDP-5,DDP-6}.

The existence and uniqueness of solutions to the initial value problem
or the initial boundary value problem of the VPFP system have been investigated
in the literature. We refer to \cite{VPFP-1,VPFP-2,VPFP-3,Hwang} for results on the classical solutions and
to \cite{VPFP-4,VPFP-5,VPFP-6,VPFP-7} for weak solutions and their regularity. Concerning the long-time
behavior of the VPFP system, we refer to \cite{time-1,time-2,time-3,Hwang}. The diffusion limit of the weak solution to the VPFP system has been studied
extensively in the literature (cf. \cite{FL-1,FL-2,FL-3,FL-4,FL-5,FL-6}). In
\cite{FL-1,FL-2,FL-5}, the authors  proved the convergence
of suitable solutions to the single species VPFP system towards a solution to the
drift-diffusion-Poisson model in the whole space.  In \cite{FL-6}, the authors established a global convergence result of the renormalized solution to the multiple species VPFP system towards a solution to the
Poisson-Nernst-Planck system. The spectrum structure and the optimal decay rate of the classical solution to the VPFP system were investigated in \cite{Li3}.

However, in contrast to the works on weak solution~\cite{FL-1,FL-2,FL-3,FL-4,FL-5,FL-6}, the diffusion limit  of the classical solution to the VPFP system \eqref{VPFP1}--\eqref{VPFP3i} has not been given despite of its importance.
On the other hand, the convergence rate of the weak solution to the VPFP system towards the solution to the DDP system hasn't been obtained in \cite{FL-1,FL-2,FL-3,FL-4,FL-5,FL-6}. Therefore, it is natural to establish an explicit convergence rate of the solution to the VPFP system towards its diffusion limit.

First of all, the VPFP system \eqref{VPFP1}--\eqref{VPFP3} has a stationary solution $(F_*,\Phi_*)=(M(v),0)$.
Hence, we define the perturbation of $F_{\eps}$ as
$$
 F_{\eps}=M+\sqrt{M}f_{\eps}.
$$
Then Cauchy problem of the VPFP system~\eqref{VPFP1}--\eqref{VPFP3i}  can be rewritten as
 \bgr
 \dt f_{\eps}+\frac{1}{\eps}v\cdot\Tdx f_{\eps}-\frac{1}{\eps}v\sqrt{M}\cdot\Tdx \Phi_{\eps}-\frac{1}{\eps^2}Lf_{\eps}
=\frac{1}{\eps} G(f_{\eps}),\label{VPFP4}\\
\Delta_x\Phi_{\eps}=\intr f_{\eps}\sqrt{M}dv,\label{VPFP5}\\
f_{\eps}(0,x,v)=f_0(x,v)=:\frac{F_0-M}{\sqrt{M}} ,\label{VPFP6}
\egr
where
 \bma
 Lf_{\eps}&= \Delta_vf_{\eps}-\frac{|v|^2}4f_{\eps}+\frac32f_{\eps},\\
 G(f_{\eps})& =\frac12 (v\cdot\Tdx\Phi_{\eps})f_{\eps}- \Tdx\Phi_{\eps}\cdot\Tdv f_{\eps} .
 \ema

Also, we define the perturbation of $N$ as
$$n=N-1.$$
Then Cauchy problem of the DDP system~\eqref{DDP1}--\eqref{DDP3} becomes
\bgr
\dt n-\Delta_x n+n=-\Tdx\cdot(n \Tdx \Phi), \label{n_1}\\
\Delta_x\Phi=n,\\
n(0,x)=n_0=\intr f_0\sqrt{M} dv. \label{n_2}
\egr

The aim of this paper  is to prove the
convergence and establish the convergence rate of strong solutions $(f_{\eps}, \Phi_{\eps})$ to \eqref{VPFP4}--\eqref{VPFP6} towards  $(n\sqrt M , \Phi)$, where $(n,\Phi)(t,x)$ is the solution of the DDP system \eqref{n_1}--\eqref{n_2}. In general, the convergence is not uniform near $t=0$. The breakdown of the uniform convergence near $t=0$ is the initial layer of the VPFP system. But we can show that if the initial data $f_0=n_0\sqrt M$,
then the uniform convergence is up to $t=0$.
\bigskip

  We denote $L^2(\R^3)$  a Hilbert space of complex-value functions $f(v)$
on $\R^3$ with the inner product and the norm
$$
(f,g)=\intr f(v)\overline{g(v)}dv,\quad \|f\|=\(\intr |f(v)|^2dv\)^{1/2}.
$$

The nullspace of the operator $L$, denoted by $N_0$, is a subspace spanned by $\sqrt{M}$.
Let $P_{0}$ be the projection operators from $L^2(\R^3_v)$
to the subspace $N_0$ with
\be
 P_{0}f=(f,\sqrt M)\sqrt M,   \quad P_1=I-P_{0}. \label{Pdr}
 \ee

Corresponding to the linearized operator $L$, we define the following dissipation norm:
\be \|f\|^2_{L^2_\sigma}=\|\Tdv f\|^2+\|v f\|^2 . \ee
This norm is stronger than the $L^2$-norm because
\be \|f\|^2_{L^2_\sigma}\ge 2\|\Tdv f\|\|vf\|\ge -2(\Tdv f,vf)=3\|f\|^2. \ee

From \cite{Yu}, the linearized operator $L$ is non-positive and
 locally coercive in the sense that there is a constant $0<\mu<1$ such that
\be
 (Lf,f)\le -\|P_1f\|^2,\quad  (Lf,f)\leq - \mu\|P_1f\|^2_{L^2_\sigma} .\label{L_4}
 \ee

\noindent\textbf{Notations:} \ \ Before state the main results in this paper, we list some notations. For any $\alpha=(\alpha_1,\alpha_2,\alpha_3)\in \mathbb{N}^3$ and $\beta=(\beta_1,\beta_2,\beta_3)\in \mathbb{N}^3$, denotes
$$\dxa=\pt^{\alpha_1}_{x_1}\pt^{\alpha_2}_{x_2}\pt^{\alpha_3}_{x_3},\quad \dvb=\pt^{\beta_1}_{x_1}\pt^{\beta_2}_{x_2}\pt^{\beta_3}_{x_3}.$$
The Fourier transform of $f=f(x,v)$
is denoted by
$$\hat{f}(\xi,v)=\mathcal{F}f(\xi,v)=\frac1{(2\pi)^{3/2}}\intr f(x,v)e^{-  i x\cdot\xi}dx.$$

In the following,  denote by $\|\cdot\|_{L^2_{x}(L^2_{v})}$  the norm of the function space $L^2(\R^3_x\times \R^3_v)$, and denote by $\|\cdot\|_{L^2_x}$ and $\|\cdot\|_{L^2_v}$  the norms of the function spaces $L^2(\R^3_x)$  and $L^2(\R^3_v)$ respectively. For any integer $k\ge 1$, we denote by
$\|\cdot\|_{H^k_x(L^2_v)}$ and $\|\cdot\|_{H^k_x }$ the norms of the function spaces
$H^k(\R^3_x,L^2(\R^3_v))$ and $H^k(\R^3_x)$ respectively.

\bigskip

Now we are ready to state  main results in this paper.

\begin{thm}\label{existence}
There exist small positive constants $\delta_0$ and $\eps_0$ such that  if the initial data $f_0$ satisfies $\|f_{0}\|_{H^4_x(L^2_v)} +\|(f_0,\sqrt{M})\|_{L^{1}_x}\le \delta_0$, then for any $\eps\in (0,\eps_0)$, there exists a unique function $f_{\eps}(t)$ to the VPFP system \eqref{VPFP4}--\eqref{VPFP6} satisfying
\be \|f_{\eps}(t) \|_{H^4_x(L^2_v)}+\|\Tdx\Phi_{\eps}(t) \|_{H^4_x }\le C\delta_0e^{-\frac t2},\quad t>0,  \ee
where $\Tdx\Phi_{\eps}(t)=\Tdx\Delta^{-1}_x(f_{\eps},\sqrt{M})$ and $C>0$ is a constant independent of $\eps$.

There exists a small constant $\delta_0>0$ such that if $\|f_{0}\|_{H^4_x(L^2_v)} +\|(f_0,\sqrt{M})\|_{L^{1}_x}\le \delta_0$, then the DDP system \eqref{n_1}--\eqref{n_2} admits a unique global solution $n(t)$ satisfying
\be
\|n (t)\|_{H^4_x}+\|\Tdx\Phi(t)\|_{H^4_x}\le C\delta_0 e^{-\frac t2},
\ee
where $\Tdx\Phi(t)=\Tdx\Delta^{-1}_xn(t)$ and $C>0$ is a constant.
\end{thm}


\begin{thm}\label{thm1.1}
Let $(f_{\eps},\Phi_{\eps})=(f_{\eps}(t,x,v),\Phi_{\eps}(t,x)) $ be the global solution to the VPFP system \eqref{VPFP4}--\eqref{VPFP6}, and let $(n,\Phi)=(n,\Phi)(t,x)$ be the global solution to the DDP system \eqref{n_1}--\eqref{n_2}.  Then, there exist two small positive constants $\delta_0$ and $\eps_0$ such that if the initial data $f_0$ satisfies $\|f_{0}\|_{H^4_x(L^2_v)}+\|\Tdv f_{0}\|_{H^3_x(L^2_v)}+\||v | f_{0}\|_{H^3_x(L^2_v)}+\|(f_0,\sqrt{M})\|_{L^{1}_x}\le \delta_0$, then for any $\eps\in (0,\eps_0)$, 
\be \|f_{\eps}(t)-n(t)\sqrt M\|_{H^2_x(L^2_v)}+\|\Tdx\Phi_{\eps}(t)-\Tdx\Phi(t)\|_{H^2_x }\le C\delta_0\(\eps e^{-\frac{t}{2}}+e^{-\frac{at}{\eps^2}}\),  \label{limit0}\ee
where  $a>0$ and $C>0$ are two constants independent of $\eps$. 

Moreover, if the initial data $f_0=n_0 \sqrt M\in N_0$  and   $\|n_{0}\|_{H^4_x} +\|n_0\|_{L^{1}_x}\le \delta_0$, then we have
\be
\|f_{\eps}(t)-n(t)\sqrt M\|_{H^2_x(L^2_v)}+\|\Tdx\Phi_{\eps}(t)-\Tdx\Phi(t)\|_{H^2_x }\le C\delta_0\eps e^{-\frac{t}{2}} . \label{limit_1a}
\ee
\end{thm}

The results in Theorem~\ref{thm1.1} on the convergence rate  of diffusion limits of the VPFP system  is proved based on the spectral analysis \cite{Li3} and the ideas inspired by \cite{Ukai1,Ellis,Li4}. Indeed, we first develop a  non-local implicit function theorem to show the existence of the eigenvalue $\lambda_0(|\xi|,\eps)$ of the linear VPFP operator $B_{\eps}(\xi)$ defined by
\eqref{B3} for $\eps|\xi|$ small and establish the expansions of the eigenvalue $\lambda_0(|\xi|,\eps)$ and the corresponding eigenfunction $\psi_0(\xi,\eps)$ as (See Lemma \ref{eigen_1} and Theorem \ref{spect3})
\be \label{lam0}
\left\{\bln
\lambda_0(|\xi|,\eps)&= -\eps^2(1+|\xi|^2)+O(\eps^4(1+|\xi|^2)^2),\\
P_0\psi_0(\xi,\eps)&= \frac{|\xi|}{\sqrt{1+|\xi|^2}} \sqrt{M} +O(\eps |\xi|),\\
P_1\psi_0(\xi,\eps)&= -i\eps \sqrt{1+|\xi|^2}  \(v\cdot\frac{\xi}{|\xi|}\)\sqrt{M}+O(\eps^2(1+|\xi|^2)).
\eln\right.
\ee
Based on the spectral analysis of $B_{\eps}(\xi)$, we can decompose the semigroup $e^{ B_{\eps}(\xi)t/\eps^2}$ into the fluid part $S_1$ and the remainder part $S_2$, where the remainder part $S_2$ satisfies the decay rate $e^{-Ct/\eps^2}$, and the fluid part $S_1$ takes the form (See Theorem \ref{rate1})
\be
S_1(t,\xi,\eps)f= e^{\frac{t}{\eps^2}\lambda_0(|\xi|,\eps)}\(f,\overline{\psi_0(\xi,\eps)}\)_\xi\psi_0(\xi,\eps).\label{lam1}
\ee
Then by applying the expansion \eqref{lam0} to the fluid part \eqref{lam1}, we can establish the convergence and the optimal convergence rate of the solution to the linear VPFP system towards its first and second order fluid limits, where the first order fluid limit is the solution to the linear DDP system (See Lemmas \ref{fl1} and \ref{fl2}).

Finally, making use of the convergence results on the fluid limits of the linear VPFP system, we can  prove the convergence and establish the optimal convergence rate of  the nonlinear VPFP system towards the DDP system  after a careful and tedious computation. Moreover, we can obtain the precise estimation on the initial layer.

The rest of this paper will be organized as follows.
In Section~\ref{sect2}, we give the spectrum analysis of the linear operator related to the linearized VPFP system. In Section~\ref{sect3}, we  establish  the first and second order fluid approximations of the solution to the  linearized VPFP system based on the spectral analysis.
In Section~\ref{sect4}, we prove the convergence and establish the convergence rate of the global solution to the original nonlinear VPFP system towards the solution to the nonlinear DDP system.

\section{Spectral analysis}
\setcounter{equation}{0}
\label{sect2}

In this section, we are concerned with the spectral analysis of the operator $B_{\eps}(\xi)$ defined by
\eqref{B3}, which will be applied to study diffusion limit of the solution to the VPFP system \eqref{VPFP4}--\eqref{VPFP6}.

From the system \eqref{VPFP4}--\eqref{VPFP6}, we have the following linearized VPFP system:
\be\label{LVPFP1}
\left\{\bln
&\eps^2\dt f_{\eps}=B_{\eps}f_{\eps},\quad t>0,\\
 &f_{\eps}(0,x,v)=f_0(x,v),
 \eln\right.
\ee
where
\be
B_{\eps}f=Lf-\eps v\cdot\Tdx f-\eps v \cdot \Tdx(-\Delta_x)^{-1}P_0f.
\ee
Take Fourier transform to \eqref{LVPFP1} in $x$ to get
\be\label{LVPFP5}
\left\{\bln
 &\eps^2\dt \hat{f}_{\eps}= B_{\eps}(\xi)\hat f_{\eps},\quad t>0,\\
  &\hat{f}_{\eps}(0,\xi,v)=\hat{f}_0(\xi,v),
   \eln\right.
\ee
where 
\be
B_{\eps}(\xi)=L-i\eps v\cdot\xi-i\eps \frac{ v\cdot\xi}{|\xi|^2}P_{0},\quad \xi\ne 0. \label{B3}
\ee

Following \cite{Yu}, we decompose $L$ as follows:
\be \label{L_1}
\left\{\bln
 Lf&=-Af+Kf,\\
Af&=-\Delta_vf+\frac{|v|^2}4f-\frac321_{\{|v|\ge R\}}f,\\
Kf&=\frac321_{\{|v|\le R\}}f,
\eln\right.
\ee
where $R > 0$ is chosen to be large enough and $1_{\{|v|\le R\}}$ is the indicator function of the
domain $|v| \le R$. The operators $A$   and $K $ are  self-adjoint on $L^2$, and $K$ is a $A-$compact operator, that is, for any
sequences both $\{u_n\}$ and $\{A u_n\}$ bounded in $L^2_v$,  $\{Ku_n\}$ contains a
convergent subsequence  in $L^2_v$ (cf. Lemma 2.3 in \cite{Yu}).
For $R>0$ sufficiently large, there are two constants $\nu_0,\nu_1>0$ such that
\be
(Af,f)\ge \nu_1\|f\|^2_{L^2_\sigma}\ge \nu_0\|f\|^2. \label{L_3}
\ee
Also, we define
\bq
  A_{\eps}(\xi)=-A- i\eps (v\cdot\xi).  \label{Cxi}
 \eq

Introduce a weighted Hilbert space $L^2_\xi(\R^3)$ for $\xi\ne 0$
as
$$
 L^2_\xi(\R^3)=\{f\in L^2(\R^3)\,|\,\|f\|_\xi=\sqrt{(f,f)_\xi}<\infty\},
$$
with the inner product defined by
$$
 (f,g)_\xi=(f,g)+\frac1{|\xi|^2}(P_{0} f,P_{0} g).
$$

Since $P_{0}$ is a self-adjoint projection operator, it follows that
 $(P_{0} f,P_{0} g)=(P_{0} f, g)=( f,P_{0} g)$ and hence
 \bq (f,g)_\xi=(f,g+\frac1{|\xi|^2}P_{0}g)=(f+\frac1{|\xi|^2}P_{0}f,g).\label{C_1a}\eq
By
\eqref{C_1a}, we have for any $f,g\in L^2_\xi(\R^3_v)\cap D(B_{\eps}(\xi))$,
 \be
 (B_{\eps}(\xi)f,g)_\xi=(B_{\eps}(\xi) f,g+\frac1{|\xi|^2}P_{0} g)=(f,B_{\eps}(-\xi)g)_\xi. \label{L_7}
\ee
Moreover,  $B_{\eps}(\xi)$ is a dissipate operator in $L^2_\xi(\R^3)$:
 \be
 {\rm Re}(B_{\eps}(\xi)f,f)_\xi=(Lf,f)\le 0. \label{L_8}
\ee
We can regard $B_{\eps}(\xi)$ as a linear operator from the space $L^2_\xi(\R^3)$ to itself because
$$
 \|f\|^2\le \|f\|^2_\xi\le(1+|\xi|^{-2})\|f\|^2,\quad \xi\ne 0.
$$

The  $\sigma(A)$ denotes the spectrum of the operator $A$. The discrete spectrum
of $A$, denoted by $\sigma_d(A)$, is the set of all isolated eigenvalues with finite multiplicity.
The essential spectrum of $A$, $\sigma_{ess}(A)$, is the set $\sigma(A)\setminus \sigma_d
(A)$. We denote $\rho(A)$ to be the resolvent set of the operator $A.$

By \eqref{L_7},  \eqref{L_8}, \eqref{L_1} and \eqref{L_3}, we have the following two lemmas.

\begin{lem}\label{SG_1}
(1) The operator $B_{\eps}(\xi)$ generates a strongly continuous contraction semigroup on
$L^2_\xi(\R^3)$, which satisfies \bq
\|e^{tB_{\eps}(\xi)}f\|_\xi\le\|f\|_\xi, \quad \forall\ t>0,\,\,f\in
L^2_\xi(\R^3_v). \eq

(2) The operator $A_{\eps}(\xi)$ generates a strongly continuous contraction semigroup on
$L^2 (\R^3)$, which satisfies \bq
\|e^{tA_{\eps}(\xi)}f\| \le e^{-\nu_0t}\|f\| , \quad \forall\ t>0,\,\,f\in
L^2 (\R^3_v). \eq
\end{lem}

\begin{proof}
Since  the operator $B_{\eps}(\xi)$ is a densely defined closed operator on  $L^2_\xi(\R^3)$, and both $B_{\eps}(\xi)$ and $B_{\eps}(\xi)^*=B_{\eps}(-\xi)$ are dissipative on $L^2_\xi(\R^3)$ by \eqref{L_8}, it follows  that $B_{\eps}(\xi)$ generates a strongly continuous contraction
semigroup on $L^2_\xi(\R^3)$. This proves (1). Similarly, we can prove (2) by \eqref{L_3}.
\end{proof}

\begin{lem}\label{Egn}
The following conditions hold for all $\xi\ne 0$ and $\eps\in (0,1)$.
 \begin{enumerate}
\item[(1)]
$\sigma_{ess}(B_{\eps}(\xi))\subset \{\lambda\in \mathbb{C}\,|\, {\rm Re}\lambda\le -\nu_0\}$ and $\sigma(B_{\eps}(\xi))\cap \{\lambda\in \mathbb{C}\,|\, -\nu_0<{\rm Re}\lambda\le 0\}\subset \sigma_{d}(B_{\eps}(\xi))$.
\item[(2)]
 If $\lambda$ is an eigenvalue of $B_{\eps}(\xi)$, then ${\rm Re}\lambda<0$ for any $\eps\xi\ne 0$ and $ \lambda=0$ iff $\eps\xi= 0$.
 \end{enumerate}
\end{lem}

\begin{proof}
By Lemma \ref{SG_1}, $\lambda- A_{\eps}(\xi)$ is invertible for ${\rm
Re}\lambda>-\nu_0$ and hence $\sigma (A_{\eps}(\xi))\subset \{\lambda\in \mathbb{C}\,|\, {\rm Re}\lambda\le -\nu_0\}$. Since $K$ is $A_\eps(\xi)-$compact and $i(v\cdot\xi)|\xi|^{-2}P_0$ is compact, namely, $ B_{\eps}(\xi)$ is a compact perturbation of $ A_{\eps}(\xi)$, it follows that $ \sigma_{ess}(B_{\eps}(\xi))=\sigma_{ess}(A_{\eps}(\xi))$   and $\sigma (B_{\eps}(\xi))$ in the domain ${\rm Re}\lambda>-\nu_0$ consists of discrete eigenvalues  with possible accumulation
points only on the line ${\rm Re}\lambda= -\nu_0$. This proves (1). By a similar as Proposition 2.2.8 in \cite{Ukai}, we can prove (2) and  the detail is omitted for simplicity.
\end{proof}

Now denote by $T$ a linear operator on $L^2(\R^3_v)$ or
$L^2_\xi(\R^3_v)$, and we define the corresponding norms of $T$ by
$$
 \|T\|=\sup_{\|f\|=1}\|Tf\|,\quad
 \|T\|_\xi=\sup_{\|f\|_\xi=1}\|Tf\|_\xi.
$$
Obviously,
 \be
(1+|\xi|^{-2})^{-1/2}\|T\|\le \|T\|_\xi\le (1+|\xi|^{-2})^{1/2}\|T\|.\label{eee}
 \ee

First, we consider the spectrum and resolvent sets of $B_{\eps}(\xi)$ for $\eps|\xi|>r_0$ with $r_0>0$ a constant. To this end, we  decompose  $B_{\eps}(\xi)$ into
\bma
\lambda-B_{\eps}(\xi)&=\lambda-A_{\eps}(\xi)-K+i\eps\frac{ v\cdot\xi}{|\xi|^2}P_{0}\nnm\\
&=\(I-K(\lambda-A_{\eps}(\xi))^{-1}+i\eps\frac{v\cdot\xi}{|\xi|^2}P_{0}(\lambda-A_{\eps}(\xi))^{-1}\)(\lambda-A_{\eps}(\xi)).\label{B_d}
\ema
Then, we have the estimates on the right hand terms of \eqref{B_d} as follows.

\begin{lem}[\cite{Yu,Li3}]\label{LP03}
 There exists a constant  $C>0$ so that the following holds
\begin{enumerate}
\item For any $\delta>0$, if ${\rm Re}\lambda\ge -\nu_0+\delta$, we have
 \be
  \|K(\lambda-A_{\eps}(\xi))^{-1}\| \to 0, \quad {\rm as}\,\,\ |{\rm Im}\lambda|+\eps|\xi|\to \infty. \label{T_7}
 \ee
 \item  For any $\delta>0$ and $ r_0>0$,  if ${\rm Re}\lambda\ge -\nu_0+\delta$ and $|\xi|\ge r_0$, we have
 \be
    \|(v\cdot\xi)|\xi|^{-2}P_{0}(\lambda-A_{\eps}(\xi))^{-1}\|
 \leq C(\delta^{-1}+1)(r_0^{-1}+1)(|\xi|+|\lambda|)^{-1}. \label{L_9}
  \ee
\end{enumerate}
\end{lem}

By \eqref{B_d} and Lemma \ref{LP03}, we have the spectral gap
of the operator $B_{\eps}(\xi)$ for $\eps|\xi|>r_0$.

\begin{lem}[Spectral gap]\label{LP01}
Fix $\eps\in (0,1)$. For any $r_0>0$, there
exists $\alpha =\alpha(r_0)>0$ such that for  $ \eps|\xi|\ge r_0$,
\bq \sigma(B_{\eps}(\xi))\subset\{\lambda\in\mathbb{C}\,|\, \mathrm{Re}\lambda \leq-\alpha\} .\label{gap}\eq
\end{lem}

\begin{proof}
Let $\lambda\in \sigma(B_{\eps}(\xi))\cap\{\lambda\in\mathbb{C}\,|\, \mathrm{Re}\lambda\geq -\nu_0+\delta\}$ with  $\delta>0$. We first show that  $\sup_{\eps|\xi|\ge r_0}|{\rm
Im}\lambda|<+\infty$. By \eqref{eee}, \eqref{T_7} and \eqref{L_9}, there exists a large constant
$r_1=r_1(\delta)>0$ such that for
$\mathrm{Re}\lambda\geq -\nu_0+\delta$ and $\eps|\xi|\geq r_1$,
\bq
\|K(\lambda-A_{\eps}(\xi))^{-1}\|_{\xi}\leq1/4,\quad
\|\eps(v\cdot\xi)|\xi|^{-2}P_0(\lambda-A_{\eps}(\xi))^{-1}\|_{\xi}\leq1/4.\label{bound}
\eq
This implies that
$I+K(\lambda-A_{\eps}(\xi))^{-1}+i\eps(v\cdot\xi)|\xi|^{-2}P_0(\lambda-A_{\eps}(\xi))^{-1}$
is invertible on $L^2_{\xi}(\R^3_v)$, which together with
\eqref{B_d} yield that $\lambda-B_{\eps}(\xi)$ is also invertible on $L^2_{\xi}(\R^3_v)$ and
satisfies
\bq
(\lambda-B_{\eps}(\xi))^{-1}=(\lambda-A_{\eps}(\xi))^{-1}\(I-K(\lambda-A_{\eps}(\xi))^{-1}+i\eps\frac{v\cdot\xi}{|\xi|^2}P_0(\lambda-A_{\eps}(\xi))^{-1}\)^{-1} . \label{E_6}
\eq
Therefore
$\{\lambda\in\mathbb{C}\,|\,\mathrm{Re}\lambda\ge
-\nu_0+\delta\}\subset \rho(B_{\eps}(\xi))$
for $\eps|\xi|\ge r_1$.

As for $r_0\le\eps|\xi|\le r_1$, by \eqref{T_7} and \eqref{L_9} there exists
$\beta=\beta(r_0,r_1,\delta)>0$ such that if $\mathrm{Re}\lambda\geq -\nu_0+\delta$, $|\mathrm{Im}\lambda|>\beta$ and $\eps|\xi|\in [r_0,r_1]$, then
\eqref{bound} still holds and thus $\lambda-B_{\eps}(\xi)$ is invertible.
This implies that $\{\lambda\in\mathbb{C}\,|\,\mathrm{Re}\lambda\ge
-\nu_0+\delta, |\mathrm{Im}\lambda|>\beta\}\subset \rho(B_{\eps}(\xi))$ for
$r_0\le\eps|\xi|\le r_1$. Thus, we conclude
 \bq
 \sigma(B_{\eps}(\xi))
 \cap\{\lambda\in\mathbb{C}\,|\,\mathrm{Re}\lambda\ge-\nu_0+\delta\}
\subset
 \{\lambda\in\mathbb{C}\,|\,\mathrm{Re}\lambda\ge
    -\nu_0+\delta,\,|\mathrm{Im}\lambda|\le \beta \},
    \quad \eps|\xi|\ge r_0.                            \label{SpH}
 \eq

Next, we prove that $\sup_{ \eps |\xi|\ge r_0}{\rm Re}\lambda<0$. Base on the above argument, it is sufficient to prove
that   ${\rm Re}\lambda<0$ for $\eps|\xi|\in[r_0,r_1]$ and $|{\rm Im}\lambda|\le \beta$. If not, there exists
$\lambda_n\in \sigma(B_{\eps}(\xi_n))$ with $\eps|\xi_n|\in[r_0,r_1]$ and $f_n\in L^2(\R^3)$ with $\|f_n\|_{\xi_n}=1$
such that
$$Lf_n-i\eps(v\cdot\xi_n)f_n-\frac{i\eps(v\cdot\xi_n)}{|\xi_n|^2}P_0 f_n=\lambda_nf_n,\quad
\text{Re}\lambda_n\rightarrow0.$$
Taking the inner product $(\cdot,\cdot)_{\xi_n}$ between the above equation and $f_n$ and choosing the real part, we have
$${\rm Re}\lambda_n\|f_n\|_{\xi_n}= (Lf_n,f_n)\le -\mu \|P_1f_n\|^2.$$
This implies that $\lim_{n\to \infty}\|P_1f_n\|=0$ and hence $\lim_{n\to \infty}\|P_0f_n\|_{\xi_n}=1$.
Since $P_0$ is a compact operator, there exists a subsequence $n_j$ of $n$ and a function $f_0\ne 0\in N_0$ such that  $P_0f_{n_j}\to f_0$ in $L^2$ as $j\to \infty$. Thus we have $f_{n_j}\to f_0$ in $L^2$ as $j\to \infty$.
Due to the fact
that $\eps|\xi_n|\in[r_0,r_1]$, $|{\rm Im}\lambda_n|\le \beta$ and ${\rm Re}\lambda_n\to 0$,  there exists a subsequence of (still denoted by) $(\xi_{n_j},\lambda_{n_j})$, and $(\xi_0,\lambda_0)$ with $\eps|\xi_0|\in[r_0,r_1]$, ${\rm Re}\lambda_0= 0$
such that $(\xi_{n_j},\lambda_{n_j})\to (\xi_0,\lambda_0).$
It follows that $B_{\eps}(\xi_0) f_0=\lambda_0 f_0$ and thus $\lambda_0$ is an eigenvalue of
$B_{\eps}(\xi_0)$ with ${\rm Re}\lambda_0=0$, which contradicts ${\rm
Re}\lambda<0$ for $\xi\ne 0$ established by Lemma~\ref{Egn}. This proves the lemma.
\end{proof}

Then, we investigate the spectrum and resolvent sets of $B_{\eps}(\xi)$  for $\eps|\xi|\le r_0$. To this end, we decompose $\lambda-B_{\eps}(\xi)$ as
follows
\be
\lambda-B_{\eps}(\xi)=\lambda P_{0}+\lambda P_1-Q_{\eps}(\xi)+i\eps P_{0}(v\cdot\xi)P_1+i\eps P_1(v\cdot\xi)\(1+\frac1{|\xi|^2}\)P_{0},\label{Bd3}
\ee
where
\be
Q_{\eps}(\xi)=L-i\eps P_1(v\cdot\xi)P_1.\label{Qxi}
\ee

\begin{lem}\label{LP}
Let $\xi\neq0$ and $Q_{\eps}(\xi)$ defined by \eqref{Qxi}. We have
 \begin{enumerate}
\item If $\lambda\ne0$, then
\be
\bigg\|\lambda^{-1}P_1(v\cdot\xi)\(1+\frac1{|\xi|^2}\)P_{0}\bigg\|_\xi\le C(|\xi|+1)|\lambda|^{-1}.\label{S_2}
\ee

\item If $\mathrm{Re}\lambda>-1 $, then the operator $\lambda P_1-Q_{\eps}(\xi)$ is invertible on $N_0^\bot$ and satisfies
\bma
\|(\lambda P_1-Q_{\eps}(\xi))^{-1}\|&\leq(\mathrm{Re}\lambda+1 )^{-1},\label{S_3}\\
\|P_{0}(v\cdot\xi)P_1(\lambda P_1-Q_{\eps}(\xi))^{-1}P_1\|_\xi
&\leq C(\mathrm{Re}\lambda+1)^{-1} |\xi|\bigg(1+ \frac{|\lambda|}{1+\eps|\xi|}\bigg)^{-1} .\label{S_5}
\ema
\end{enumerate}
\end{lem}

\begin{proof}
By using
 $$
\bigg\|\lambda^{-1}P_1(v\cdot\xi)\(1+\frac1{|\xi|^2}\)P_{0}f\bigg\|_\xi\le C|\lambda|^{-1}\(|\xi|+\frac1{|\xi|}\)\|P_{0}f\|\le C|\lambda|^{-1}(|\xi|+1)\|f\|_\xi,
$$
we prove \eqref{S_2}.

Then, we show that for any $\lambda\in\mathbb{C}$ with
$\mathrm{Re}\lambda>-1 $, the operator $\lambda P_1-Q_{\eps}(\xi)=\lambda
 P_1-L+i\eps P_1(v\cdot\xi) P_1$ is  invertible from $N_0^\bot$ to
itself. Indeed, by \eqref{L_4}, we obtain for any $f\in N_0^\bot\cap
D(L)$ that
 \be
 \text{Re}([\lambda  P_1-L+i\eps P_1(v\cdot\xi) P_1]f,f)
 =\text{Re}\lambda(f,f)-(Lf,f)\geq(1+\text{Re}\lambda)\|f\|^2, \label{A_1}
 \ee
which implies that the operator $\lambda P_1-Q_{\eps}(\xi)$ is a one-to-one
map from $N_0^\bot$ to itself so long as $\text{Re}\lambda>-1 $, and its range $\textrm{Ran}[\lambda P_1-Q_{\eps}(\xi)]$ is a closed
subspace of $L^2(\R^3_v)$. It then remains to show that the operator
$\lambda P_1-Q_{\eps}(\xi)$ is also a surjective map from $N_0^\bot$ to
$N_0^\bot$, namely, $\textrm{Ran}[\lambda P_1-Q_{\eps}(\xi)] = N_0^\bot$.
In fact, if it does not hold, then there exists a function $g\in
N_0^\bot \setminus \textrm{Ran}[\lambda P_1-Q_{\eps}(\xi)]$ with $g\neq0$
so that for any $f\in N_0^\bot\cap D(L)$ that
$$
 ([\lambda P_1-L+ i\eps P_1(v\cdot\xi) P_1]f,g)
 =(f,[\overline{\lambda} P_1-L- i\eps P_1(v\cdot\xi) P_1]g)=0,
$$
which yields $g=0$ since the operator
$\overline{\lambda} P_1-L- i\eps  P_1(v\cdot\xi) P_1$ is dissipative and
satisfies the same estimate as \eqref{A_1}. This is a contradiction,
and thus $\textrm{Ran}[\lambda P_1-Q_{\eps}(\xi)] = N_0^\bot$.  The estimate
\eqref{S_3} follows directly from \eqref{A_1}.

By \eqref{S_3} and $\|P_{0}(v\cdot\xi) P_1f\|_\xi\le C(|\xi|+1)\|P_1f\|$, we have
 \be
 \| P_{0}(v\cdot\xi) P_1(\lambda  P_1-Q_{\eps}(\xi))^{-1} P_1f\|_\xi \leq C|\xi|(\mathrm{Re}\lambda+1 )^{-1}\|f\|.   \label{2.33a}
 \ee
Meanwhile, we can decompose the operator $ P_{0}(v\cdot\xi) P_1(\lambda
 P_1-Q_{\eps}(\xi))^{-1} P_1$ as
 $$
  P_{0}(v\cdot\xi) P_1(\lambda  P_1-Q_{\eps}(\xi))^{-1} P_1=\frac1\lambda  P_{0}(v\cdot\xi) P_1+\frac1\lambda  P_{0}(v\cdot\xi) P_1Q_{\eps}(\xi)(\lambda
 P_1-Q_{\eps}(\xi))^{-1} P_1.
 $$
This together with \eqref{S_3} and the fact
 $
 \| P_{0}(v\cdot\xi) P_1 Q_{\eps}(\xi)f\|_\xi\leq
C|\xi|(1+\eps|\xi|)\|P_1f\|
 $
give
 \be
 \| P_{0}(v\cdot\xi) P_1(\lambda  P_1-Q_{\eps}(\xi))^{-1} P_1f\|_\xi
 \leq
C|\xi||\lambda|^{-1}[(\mathrm{Re}\lambda+1 )^{-1}+1](1+\eps|\xi|)\|f\|. \label{2.33}
 \ee
The combination of the two cases \eqref{2.33a} and \eqref{2.33} yields \eqref{S_5}.
\end{proof}

By \eqref{Bd3} and Lemmas~\ref{Egn}--\ref{LP}, we are able to analyze  the spectral
and resolvent sets of the operator $B_{\eps}(\xi)$ as follows.

\begin{lem}\label{spectrum2} Fix $\eps\in (0,1)$. The following facts hold.
 \begin{enumerate}
\item  For all $\xi\ne 0$, there exists $y_0>0$ such that
\bq
 \{\lambda\in\mathbb{C}\,|\,
     \mathrm{Re}\lambda\ge-\frac{\nu_0}{2},\,|\mathrm{Im}\lambda|\geq y_0\}
 \cup\{\lambda\in\mathbb{C}\,|\,\mathrm{Re}\lambda>0\}
 \subset\rho(B_{\eps}(\xi)).                           \label{rb1}
\eq

\item  For any $\delta>0$, there exists $r_0=r_0(\delta)>0$ such that if $\eps(1+|\xi|)\leq r_0$, then
 \bq
 \sigma(B_{\eps}(\xi))\cap\{\lambda\in\mathbb{C}\,|\,\mathrm{Re}\lambda\ge-\frac12\}
 \subset
 \{\lambda\in\mathbb{C}\,|\,|\lambda|\le\delta\}.   \label{sg4}
 \eq
\end{enumerate}
\end{lem}

\begin{proof}
By Lemma \ref{LP}, we have for $\rm{Re}\lambda>-1 $ and
$\lambda\neq 0$  that the operator
$\lambda  P_0+\lambda  P_1-Q_{\eps}(\xi)$ is invertible on
$L^2_\xi(\R^3_v)$ and it satisfies
 \be
  (\lambda  P_0+\lambda  P_1-Q_{\eps}(\xi))^{-1} =\lambda^{-1} P_0+(\lambda  P_1-Q_{\eps}(\xi))^{-1} P_1,
 \ee
because the operator $\lambda  P_0$ is orthogonal to $\lambda
 P_1-Q_{\eps}(\xi)$. Therefore, we can re-write \eqref{Bd3} as
\bmas
 \lambda-B_{\eps}(\xi)
=&(I+Y_{\eps}(\lambda,\xi))(\lambda  P_0+\lambda  P_1-Q_{\eps}(\xi)),
 \\
Y_{\eps}(\lambda,\xi)= & i\eps\lambda^{-1} P_1(v\cdot\xi)(1+\frac1{|\xi|^2}) P_0
    +i\eps P_0(v\cdot\xi) P_1(\lambda P_1-Q_{\eps}(\xi))^{-1} P_1.
 \emas

By Lemma \ref{LP01}, there exists $R_0>0$ large enough so that if $\mathrm{Re}\lambda\geq
-\nu_0/2$ and $\eps |\xi|\geq R_0$, then   $\lambda-B_{\eps}(\xi)$ is  invertible on $L^2_{\xi}(\R^3_v)$ and
satisfies \eqref{E_6}. Thus, $\{\lambda\in\mathbb{C}\,|\,\mathrm{Re}\lambda\ge
-\nu_0/2\}\subset \rho(B_{\eps}(\xi))$ for $\eps |\xi|\ge R_0$.

For the case $\eps|\xi|\leq R_0$ such that $\eps(1+|\xi|)\leq R_0+1$, by \eqref{S_2} and
\eqref{S_5} we can choose $y_0>0$ such that it holds
for $\mathrm{Re}\lambda\ge-1/2$ and
$|\mathrm{Im}\lambda|\geq y_0$ that
 \be
 \|\eps\lambda^{-1}P_1(v\cdot\xi)(1+|\xi|^{-2})P_{0}\|_\xi\leq \frac14,
 \quad
\|\eps P_{0}(v\cdot\xi)P_1(\lambda P_1-Q_{\eps}(\xi))^{-1}P_1\|_\xi\leq\frac14.\label{bound_1}
 \ee
This implies that the operator $I+Y_{\eps}(\lambda,\xi)$ is invertible on
$L^{2}_\xi(\R^3_v)$ and thus  $\lambda-B_{\eps}(\xi)$ is invertible on $L^{2}_\xi(\R^3_v)$ and satisfies
 \be
 (\lambda-B_{\eps}(\xi))^{-1}
 =(\lambda^{-1} P_0+(\lambda P_1-Q_{\eps}(\xi))^{-1} P_1)(I+Y_{\eps}(\lambda,\xi))^{-1}.\label{S_8}
 \ee
Therefore, $\rho(B_{\eps}(\xi))\supset \{\lambda\in\mathbb{C}\,|\,{\rm
Re}\lambda\ge-1/2, |{\rm Im}\lambda|\ge y_0\}$ for $\eps|\xi|\leq R_0$. This and Lemma~\ref{SG_1} lead to \eqref{rb1}.

Assume that $ |\lambda|>\delta$ and $\mathrm{Re}\lambda\ge-1/2$. Then, by \eqref{2.33a}  and
\eqref{S_5} we can choose $r_0=r_0(\delta)>0$ so that estimates \eqref{bound_1}
still hold for $\eps(1+|\xi|)\leq r_0$, and the operator
$\lambda-B_{\eps}(\xi)$ is invertible on $L^{2}_\xi(\R^3)$.
Therefore, we have
 $\rho(B_{\eps}(\xi))\supset\{\lambda\in\mathbb{C}\,|\, |\lambda|>\delta,\mathrm{Re}\lambda\ge-1/2\}$
for $\eps(1+|\xi|)\leq r_0$, which gives \eqref{sg4}.
\end{proof}

Now we  establish  the asymptotic expansions of the eigenvalues  and  eigenfunctions of $B_{\eps}(\xi)$ for $\eps|\xi|$ sufficiently small.
Firstly, we consider a 1-D eigenvalue problem:
\be
B_{\eps}(s)e=:\(L-i\eps s v_1-i\eps\frac{v_1}{s}P_{0}\)e=\beta e,\quad s\in \R.\label{L_2}
\ee

Let $e$ be the eigenfunction of \eqref{L_2}, we rewrite $e$ in the
form $e=e_0+e_1$, where $e_0=P_{0}e=C_0\sqrt M$ and $e_1=(I-P_{0})e=P_1e$.  The
eigenvalue problem \eqref{L_2} can be decomposed into
\bma
&\lambda e_0=-i\eps sP_{0}[ v_1(e_0+e_1)],\label{A_2}\\
&\lambda e_1=Le_1-i\eps sP_1[ v_1(e_0+e_1)]-i\eps \frac{v_1}{s}e_0.\label{A_3}
\ema

From Lemma \ref{LP} and \eqref{A_3}, we obtain that for any
$\text{Re}\lambda>-1 $,
\bq e_1=i\eps (L-\lambda P_1-i sP_1v_1P_1)^{-1}P_1\(sv_1e_0+\frac{v_1}{s}e_0\).\label{A_4}\eq

Substituting  \eqref{A_4} into \eqref{A_2} and taking inner product the resulted equation with $\sqrt M$ gives
\be
  \lambda C_0=\eps^2(1+s^2)(R(\lambda,\eps s)v_1\sqrt{M},v_1\sqrt{M})C_0.\label{eigen}
\ee
where
$$ R(\lambda,s)=(L-\lambda P_1-i sP_1v_1P_1)^{-1}.$$

Denote
\be
D_0(z,s,\eps)=z- (1+s^2)(R(\eps^2z,\eps s)v_1\sqrt{M},v_1\sqrt{M}).\label{ddd}
\ee

\begin{lem}\label{eigen_1} There are  two constants $r_0,r_1>0$ such that the equation $D_0(z,s,\eps)=0$ has a unique solution $z=z(s,\eps)$ for  $\eps(1+|s|)\le r_0$ and $|z+1+s^2|\le r_1(1+s^2) $, which is a $C^{\infty}$ function of $s$, $\eps$ and satisfies
\be
z(s,0)=-(1+s^2) ,\quad \pt_\eps z(s,0)=0. \label{z1}
\ee
In particular, $z(s,\eps)$ satisfies the following expansion:
\be
z(s,\eps)=-(1+s^2)+O(\eps^2(1+s^2)^2). \label{z2}
\ee
\end{lem}
\begin{proof}
By \eqref{ddd} and the fact $L(v_j\sqrt{M})=-v_j\sqrt{M}$, $j=1,2,3$, we have
\be D_0(z(s),s,0)=0\quad {\rm with}\quad z(s)=-(1+s^2).  \label{a}\ee
Define
$$D(z,s,\eps)=(1+s^2)R_{11}(\eps^2z,\eps s),$$
with $R_{11}(\eps^2z,\eps s)= (R(\eps^2z,\eps s)v_1\sqrt{M},v_1\sqrt{M})$.
 It is straightforward to verify that a solution of $D_0(z,s,\eps)=0$ for any fixed $s$ and $\eps$ is a fixed point of $D(z,s,\eps)$.

Since for any $(z,s,\eps)\in \mathbb{C}\times\R\times \R$,
$$ |\partial_z R_{11}(\eps^2 z,\eps s)|\le C\eps^2 ,\quad |\partial_\eps R_{11}(\eps^2 z,\eps s)|\le C(|\eps z|+|s|),$$
 it follows that
\bmas
|D(z,s,\eps)-z(s)|&= (1+s^2) \Big|R_{11}(\eps^2z,\eps s )-R_{11}(0,0 )\Big|\nnm\\
&\le C(1+s^2)(|\eps^2z|+|\eps s|)\le r_1(1+s^2),
\\
|D(z_1,s,\eps)-D(z_2,s,\eps)|&\le C\eps^2(1+s^2)|z_1-z_2|\le \frac12|z_1-z_2|,
\emas
for $|z-z(s)|\le r_1(1+s^2)$ and $\eps(1+|s|)\le r_0$ with $r_0,r_1>0$ sufficiently small.

Hence by the contraction mapping theorem, there exists a unique fixed point $z(s,\eps): \eps\in B(0,r_0/(1+|s|))\to B(z(s),r_1(1+s^2))$ such that $D(z(s,\eps),s, \eps)=z(s,\eps)$  and $z(s,0)=z(s)$. This is equivalent to that $D_0(z(s,\eps),s,\eps)=0$. Since $D_0(z,s,\eps)$ is $C^{\infty}$ with respect to $z$, $s$ and $\eps$, it follows that $z(s,\eps)$ is a $C^{\infty}$ function with respect to $s$ and $\eps$.

Next, we estimate the derivative of $z(s,\eps)$. By \eqref{ddd} and using the fact
$$
L\sqrt{M}=0,\,\,\, L(v_j\sqrt{M})=-v_j\sqrt{M},\,\,\, LP_1(v^2_j\sqrt{M})=-2P_1(v^2_j\sqrt{M}),\,\,\,j=1,2,3,
$$
with $P_1(v^2_j\sqrt{M})=(v^2_j- |v|^2/3)\sqrt{M}$, we have
\bmas
\pt_z D_0(z,s,0)&=1,\\
\pt_\eps D_0(z,s,0)&=i(1+s^2)s(P_1(v_1^2\sqrt{M}),v_1\sqrt{M})=0.
\emas
It follows that
\be
 \pt_\eps z(s,0)=- \frac{{\partial_\eps}D_0(z(s),s,0)}{{\partial_z}D_0(z(s),s,0)} =0. \label{b}
\ee
Combining \eqref{a}--\eqref{b}, we prove \eqref{z1}. Finally, by a direct computation  we obtain
$$
 \partial_z D_0(z,s,\eps) =1 +O(1)\eps^2(1+s^2) ,\quad  \partial_{\eps} D_0(z,s,\eps) =O(1)\eps(1+s^2)^2,
$$
for $\eps(1+|s|)\le r_0$ and $|z+1+s^2|\le r_1(1+s^2) $. This implies that
$$
 \pt_\eps z(s,\eps)=- \frac{{\partial_\eps}D_0(z,s,\eps)}{{\partial_z}D_0(z,s,\eps)} =O(1)\eps(1+s^2)^2,\quad \eps(1+|s|)\le r_0,
$$
which together with \eqref{z1} leads to \eqref{z2}. This proves the lemma.
\end{proof}

With the help of Lemma \ref{eigen_1}, we have the eigenvalue $\beta_0(s,\eps)$ and the corresponding eigenfunction $e_0(s,\eps)$ of $B_{\eps}(s)$
defined by \eqref{L_2} as follows.

\begin{lem}\label{eigen_4}
(1) There exists a small constant $r_0>0$ such that $ \sigma(B_{\eps}(s))\cap \{\lambda\in \mathbb{C}\,|\, \mathrm{Re}\lambda>-1 /2\} $ consists of one point $\beta_0(s,\eps)$ for   $\eps(1+|s|)\le  r_0$. The eigenvalue $\beta_0(s,\eps)$  is a $C^\infty$ function of $s$ and $\eps $, and admit the following asymptotic expansion for $\eps(1+|s|)\le  r_0$:
 \be
 \beta_{0}(s,\eps) = -\eps^2(1+s^2)+O(\eps^4(1+s^2)^2). \label{specr0}
 \ee

(2) The corresponding eigenfunction $e_0(s,\eps)=e_0(s,\eps,v)$ is $C^\infty$  in $s$ and $\eps$ satisfying
 \be            \label{eigf2}
  \left\{\bln
P_0e_0(s,\eps)&= \frac{s}{\sqrt{1+s^2}} \sqrt{M} +O(\eps s),\\
P_1e_0(s,\eps)&=  -i\eps \sqrt{1+s^2}  v_1\sqrt{M}+O(\eps^2(1+s^2)).
  \eln\right.
  \ee
\end{lem}

\begin{proof}
The eigenvalue $\beta_0(s,\eps)$ and the eigenfunction $e_0(s,\eps)$ can be constructed as follows. We take $\beta_0=\eps^2 z(s,\eps)$ with $ z(s,\eps)$ being the solution of the equation  $D_0(z,s,\eps)=0$ given in Lemma \ref{eigen_1}, and define the corresponding eigenfunction $e_0(s,\eps)$  by
 \be
  e_0(s,\eps)=a_0(s,\eps)\sqrt{M}   +i \eps\(s+\frac1s\) a_0(s,\eps)(L-\eps^2z(s,\eps)-i s\eps P_1v_1P_1)^{-1} v_1\sqrt{M},  \label{C_2}
\ee
where $a_0(s,\eps)$ is a complex value function determined later.
We can normalize $e_0(s,\eps)$ by taking
$$\(e_0(s,\eps),\overline{e_0(s,\eps)}\)_s=(e_0,\overline{e_0})+\frac{1}{s^2} (P_0e_0,P_0\overline{e_0})=1.$$

The coefficient $a_0(s,\eps) $ is determined by the normalization condition  as
 \bq \label{C_4}
 a_0(s,\eps)^2\bigg(1+\frac1{s^2}+\eps^2\bigg(s+\frac1s\bigg)^2D_1(s,\eps)\bigg)=1,
 \eq
 where $D_1(s,\eps)=(R(\beta_0,\eps s) v_1\sqrt{M}, R(\overline{\beta_0},-\eps s) v_1\sqrt{M}).$ Substituting \eqref{specr0} into \eqref{C_4}, we obtain
 \be a_0(s,\eps)=\frac{s}{\sqrt{1+s^2}}\[1+O(\eps \sqrt{1+s^2})\]. \label{C_3}\ee

Combining \eqref{C_2} and \eqref{C_3},
we can obtain the expansion of $e_0(s,\eps)$ given in \eqref{eigf2}. This completes the proof of the lemma.
\end{proof}

Now we consider a 3-D eigenvalue problem:
\bq B_{\eps}(\xi)\psi=\(L- i\eps v\cdot\xi-i\eps\frac{ v\cdot\xi}{|\xi|^2}P_{0}\)\psi=\lambda \psi,\quad \xi\in \R^3.\label{L_3a}\eq

With the help of Lemma \ref{eigen_4}, we have the eigenvalue $\lambda_0(|\xi|,\eps)$ and the corresponding eigenfunction $\psi_0(\xi,\eps)$ of $B_{\eps}(\xi)$  defined by \eqref{L_3a} as follows.

\begin{thm}\label{spect3}
(1) There exists a small constant $r_0>0$ such that $ \sigma(B_{\eps}(\xi))\cap \{\lambda\in \mathbb{C}\,|\, \mathrm{Re}\lambda>-1 /2\} $ consists of one point $\lambda_0(|\xi|,\eps)$ for   $\eps(1+|\xi|)\le  r_0$.  The eigenvalue $\lambda_0(|\xi|,\eps)$  is a $C^\infty$ function of $|\xi|$ and $\eps$ for $\eps(1+|\xi|)\leq r_0$:
 \be
 \lambda_0(|\xi|,\eps)=\beta_0(|\xi|,\eps)= -\eps^2(1+|\xi|^2)+O(\eps^4(1+|\xi|^2)^2),\label{eigen_2}
 \ee
 where $\beta_0(s,\eps)$ is a $C^\infty$ function of $s$ and $\eps$ given in Lemma \ref{eigen_4}.

(2) The eigenfunction $\psi_0(\xi,\eps)=\psi_0(\xi,\eps,v)$ satisfies
 \be
  \left\{\bln                      \label{eigf1}
P_0\psi_0(\xi,\eps)&= \frac{|\xi|}{\sqrt{1+|\xi|^2}} \sqrt{M} +O(\eps |\xi|),\\
P_1\psi_0(\xi,\eps)&= -i\eps \sqrt{1+|\xi|^2}  \(v\cdot\frac{\xi}{|\xi|}\)\sqrt{M}+O(\eps^2(1+|\xi|^2)).
  \eln\right.
  \ee
\end{thm}

\begin{proof}
Let $\O$ be a rotational transformation  in $\R^3$ such that $\O:\frac{\xi}{|\xi|}\to(1,0,0)$. We have
\bq \O^{-1} \(L- i\eps v\cdot\xi-i\eps\frac{ v\cdot\xi}{|\xi|^2}P_{0}\)\O=L- i\eps sv_1-i\eps\frac{  v_1}{s}P_{0}.\eq
From Lemma \ref{eigen_4}, we have the following eigenvalue and eigenfunction for \eqref{L_3a}:
\bgrs
\(L- i\eps v\cdot\xi-i\eps\frac{ v\cdot\xi}{|\xi|^2}P_{0}\)\psi_0(\xi,\eps)=\lambda_0(|\xi|,\eps) \psi_0(\xi,\eps),\\
\lambda_0(|\xi|,\eps)=\beta_0(|\xi|,\eps),\quad \psi_0(\xi,\eps)=\O e_0(|\xi|,\eps).
\egrs
This proves the theorem.
\end{proof}

 By virtue of Lemmas~\ref{LP03}--\ref{spectrum2} and Theorem \ref{spect3}, we can make a detailed analysis on the semigroup $S(t,\xi,\eps)=e^{\frac{t}{\eps^2}B_{\eps}(\xi)}$ in terms of an  argument similar to that of Theorem~3.4 in \cite{Li2}, we give the details of the proof in Appendix.

\begin{thm}\label{rate1}
Let $r_0>0$ be given in Theorem \ref{spect3}. For any fixed $\eps\in(0,r_0)$, the semigroup $S(t,\xi,\eps)=e^{\frac{t}{\eps^2}B_{\eps}(\xi)}$ with $\xi\neq0$ can be decomposed into
\bq  S(t,\xi,\eps)f=S_1(t,\xi,\eps)f+S_2(t,\xi,\eps)f,\quad f\in L^2_\xi(\R^3_v), \ \ t>0,\label{B_0}\eq
where
\be
S_1(t,\xi,\eps)f=e^{\frac{1}{\eps^2}\lambda_0(|\xi|,\eps)t}\(f,\overline{\psi_0(\xi,\eps)}\)_\xi \psi_0(\xi,\eps)1_{\{\eps (1+|\xi|)\le r_0\}}, \label{S1}
\ee
with $(\lambda_0(|\xi|,\eps),\psi_0(\xi,\eps))$ being the eigenvalue and eigenfunction of the operator $B_{\eps}(\xi)$ given in Theorem \ref{spect3} for $\eps(1+|\xi|)\le r_0$,
and $S_2(t,\xi,\eps) $ satisfies for two constants $a>0$ and $C>0$ independent of $\xi$ and $\eps$ that
\be
\|S_2(t,\xi,\eps)f\|_\xi\le Ce^{-\frac{at}{\eps^2}}\|f\|_\xi. \label{S2}
\ee
\end{thm}

\section{Fluid approximations of semigroup}
\setcounter{equation}{0}
\label{sect3}
In this section,  we give the first and second order fluid approximations of the semigroup $e^{\frac{t}{\eps^2}B_\eps}$, which will be used to prove the convergence and establish the convergence rate of the solution to  the VPFP system towards the solution to the  DDP system.

Let us introduce a Sobolev space of function $f=f(x,v)$ by $ H^k_P=\{f\in L^2(\R^3_{x}\times \R^3_v)\,|\, \|f\|_{H^k_P}<\infty\,\}\ (L^2_P=H^0_P)$
with the norm $\|\cdot\|_{H^k_P}$ defined by
 \bma
 \|f\|_{H^k_P}
 &=\(\intr (1+|\xi|^2)^k\|\hat{f}\|^2_\xi d\xi \)^{1/2}\nnm\\
 &=\(\intr (1+|\xi|^2)^k \(\|\hat{f}\|^2+\frac1{|\xi|^2}\lt|(\hat{f},\sqrt{M})\rt|^2\)d\xi
    \)^{1/2}, \label{H2P}
 \ema
where $\hat{f}=\hat{f}(\xi,v)$ denotes the Fourier transform of $f(x,v)$ with respect to $x\in \R^3$. Note that it holds
$$
 \|f\|^2_{H^k_P}
 =\|f\|^2_{H^k_{x}(L^2_{v})}+\| \nabla_x\Delta_x^{-1}(f,\sqrt{M})\|^2_{H^k_{x}}.
$$

For any $f_0\in L^2(\R^3_{x}\times \R^3_v)$, set
 \be
  e^{\frac{t}{\eps^2}B_\eps}f_0=(\mathcal{F}^{-1}e^{\frac{t}{\eps^2}B_\eps(\xi)}\mathcal{F})f_0.
  \ee
By Lemma \ref{SG_1}, we have
 $$
 \|e^{\frac{t}{\eps^2}B_\eps}f_0\|_{H^k_P}=\intr (1+|\xi|^2)^k\|e^{\frac{t}{\eps^2}B_\eps(\xi)}\hat f_0\|^2_\xi d\xi\le \intr (1+|\xi|^2)^k\|\hat f_0\|^2_\xi d\xi
=\|f_0\|_{H^k_P}.
 $$
This means that the linear operator $\eps^{-2} B_\eps$ generates a strongly continuous contraction semigroup $e^{\frac{t}{\eps^2}B_\eps}$ in $H^k_P$, and therefore, $f(t,x,v)=e^{\frac{t}{\eps^2}B_\eps}f_0$ is a global solution to the linearized VPFP system \eqref{LVPFP1} for any $f_0\in H^k_P$.

Now we are going to establish the first and second order fluid approximations of the semigroup $e^{\frac{t}{\eps^2}B_\eps}$. First of all, we introduce the following linearized DDP system to \eqref{n_1}--\eqref{n_2}:
\be
\dt n=Dn,\quad n(0,x)=n_0(x),\label{LDDP}
\ee
where
\be  D=-I+\Delta_x . \label{A}\ee
It is easy to verify that the operator $D$ generates a strongly continuous contraction semigroup $e^{tD}$ in $H^k_x$. Thus, $n(t,x)=e^{tD}n_0$ is a global solution to the linearized DDP system \eqref{LDDP} for any $n_0\in H^k_x$.

We give a prepare lemma which will be used to study the  fluid dynamical approximations of the semigroup $e^{\frac{t}{\eps^2}B_\eps}$.

\begin{lem} \label{S2a}
For any $f_0\in N_0$, we have
\be
\|S_{2}(t,\xi,\eps)f_0\|_{\xi}\le C\(\eps(1+|\xi|)1_{\{\eps(1+|\xi|)\le r_0\}}+1_{\{\eps(1+|\xi|)\ge r_0\}}\)e^{-\frac{at}{\eps^2}}\| f_0\|_{\xi} .\label{S5}
\ee
\end{lem}

\begin{proof}
Define the projection $P_{\eps}(\xi)$ by
$$P_{\eps}(\xi)f= \(f,\overline{\psi_0(\xi,\eps)}\)_{\xi} \psi_0(\xi,\eps),\quad \forall f\in L^2(\R^3_v),$$
where $ \psi_0(\xi,\eps)$ is the eigenfunctions of $B_{\eps}(\xi)$ for $\eps(1+|\xi|)\le r_0$ defined by \eqref{eigf1}.

By Theorem \ref{rate1}, we can assert that
\be S_1(t,\xi,\eps)=S(t,\xi,\eps)1_{\{\eps(1+|\xi|)\le r_0\}}P_{\eps}(\xi). \label{S1a}
\ee
Indeed, it follows from \eqref{V_3} that for $\eps(1+|\xi|)\le r_0$,
 \bmas
  e^{\frac{t}{\eps^2}B_{\eps}(\xi)}P_{\eps}(\xi)f =&\frac1{2\pi i}\int^{\kappa+ i\infty}_{\kappa- i\infty}
   e^{ \frac{\lambda t}{\eps^2}}(\lambda-B_{\eps}(\xi))^{-1}P_{\eps}(\xi)fd\lambda\\
   =&\frac1{2\pi i}\int^{\kappa+ i\infty}_{\kappa- i\infty}
   e^{ \frac{\lambda t}{\eps^2}}(\lambda-\lambda_0(|\xi|,\eps))^{-1}P_{\eps}(\xi)fd\lambda\\
   =&e^{\frac{1}{\eps^2}\lambda_0(|\xi|,\eps)t}(f,\overline{\psi_0(\xi,\eps)})_\xi \psi_0(\xi,\eps)=S_1(t,\xi,\eps)f   .
 \emas
By \eqref{S1a}, we can decompose $S_2(t,\xi,\eps)$ into
\be
S_2(t,\xi,\eps)=S_{21}(t,\xi,\eps)+S_{22}(t,\xi,\eps),\label{S3}
\ee
where
\bma
S_{21}(t,\xi,\eps)&=S(t,\xi,\eps)1_{\{\eps(1+|\xi|)\le r_0\}}\(I- P_{\eps}(\xi)\),\\
S_{22}(t,\xi,\eps)&=S(t,\xi,\eps)1_{\{\eps(1+|\xi|)\ge r_0\}}.
\ema
It holds that
\be
\|S_{2j}(t,\xi,\eps)f\|_{\xi}\le Ce^{-\frac{at}{\eps^2}}\| f\|_{\xi},\quad j=1,2.
\ee

Since  for any $f_0\in N_0$,
$$
f_0-P_{\eps}(\xi)f_0
=(f_0,h_0(\xi))_\xi h_0(\xi)-(f_0,\overline{\psi_0(\xi,\eps)})_\xi \psi_0(\xi,\eps),
$$
where $h_0(\xi)=\frac{|\xi|}{\sqrt{1+|\xi|^2}} \sqrt{M}$, it follows from \eqref{eigf1} and the the fact $S_{21}(t,\xi,\eps)=S_{21}(t,\xi,\eps)(I- P_{\eps}(\xi))$ that
\be
\|S_{21}(t,\xi,\eps)f_0\|_{\xi}\le C\eps(1+|\xi|)1_{\{\eps(1+|\xi|)\le r_0\}}e^{-\frac{at}{\eps^2}}\|f_0\|_{\xi}.\label{S4}
\ee
By combining \eqref{S3}--\eqref{S4}, we obtain \eqref{S5}.
\end{proof}

Then, we have the first order fluid approximation of the semigroup $e^{\frac{t}{\eps^2}B_\eps}$ as follow.

\begin{lem} \label{fl1}
For any integer $k\ge 0$ and any $f_0\in H^{k+1}_{P}$, we  have
\be
 \left\|e^{\frac{t}{\eps^2}B_\eps}f_0-e^{tD}n_0\sqrt{M}\right\|_{H^k_{P}}
 \le C\(\eps e^{-\frac{t}{2}}+e^{-\frac{at}{\eps^2}}\)\|f_0\|_{H^{k+1}_{P}}, \label{limit1}
\ee
where $n_0=(f_0,\sqrt{M}),$ and $a$, $C>0$ are two constants independent of $\eps$. Moreover, if $f_0\in N_0$, then
\be
 \left\|e^{\frac{t}{\eps^2}B_\eps}f_0-e^{tD}n_0\sqrt{M}\right\|_{H^k_{P}}
 \le C\eps e^{-\frac{t}{2}}\|f_0\|_{H^{k+1}_{P}}. \label{limit1a}
\ee
\end{lem}

\begin{proof}
For simplicity, we only prove the case of $k=0$. By Theorem \ref{rate1} and taking $\eps\le r_0/2$ with $r_0>0$ given in Theorem \ref{spect3}, we have
\bma
\left\|e^{\frac{t}{\eps^2}B_\eps}f_0-e^{tD}n_0\sqrt{M}\right\|^2_{L^2_{x}(L^2_{v})}=&\intr \left\|e^{\frac{t}{\eps^2}B_{\eps}(\xi)}\hat{f}_0-e^{tD(\xi)}\hat{n}_0\sqrt{M}\right\|^2d\xi \nnm\\
\le &3\int_{1+|\xi|\le \frac{r_0}{\eps}} \left\|S_1(t,\xi,\eps)\hat{f}_0-e^{tD(\xi)}\hat{n}_0\sqrt{M}\right\|^2d\xi\nnm\\
&+3\intr \left\|S_2(t,\xi,\eps)\hat{f}_0\right\|^2d\xi+3\int_{1+|\xi|\ge \frac{r_0}{\eps}}  \left|e^{tD(\xi)}\hat{n}_0\right|^2d\xi\nnm\\
=&I_1+I_2+I_3, \label{S_4}
\ema
where $D(\xi)=-(1+|\xi|^2)$ is the Fourier transform of the operator $D$.

We estimate $I_j$, $j=1,2,3$ as follows. Since from \eqref{S1}, \eqref{eigen_2} and \eqref{eigf1},
$$
S_1(t,\xi,\eps)\hat{f}_0=e^{-(1+|\xi|^2)t+O(\eps^2(1+|\xi|^2)^2)t}\(\hat{n}_0\sqrt{M}+O(\eps\sqrt{1+|\xi|^2})\),
$$
 it follows that
\bma
I_1=& 3\int_{1+|\xi|\le \frac{r_0}{\eps}} \Big\|e^{-(1+|\xi|^2)t+O(\eps^2(1+|\xi|^2)^2)t}\(\hat{n}_0 \sqrt{M}+O(\eps\sqrt{1+|\xi|^2})\) \nnm\\
&\qquad-e^{-(1+|\xi|^2)t }\hat{n}_0\sqrt{M}\Big\|^2d\xi\nnm\\
\le & C\eps^2\int_{1+|\xi|\le \frac{r_0}{\eps}}e^{-(1+|\xi|^2)t}\( r^2_0(1+|\xi|^2)^3t^2|\hat{n}_0|^2
 +(1+|\xi|^2)\|\hat{f}_0\|^2\) d\xi\nnm\\
\le& C\eps^2e^{-t}\|f_0\|^2_{H^1_x(L^2_{v})}.\label{I1}
\ema

By \eqref{S2}, we have
\be
I_2\le C\intr e^{-2\frac{at}{\eps^2}}\|\hat{f}_0\|^2_\xi d\xi
\le Ce^{-2\frac{at}{\eps^2}}\(\|f_0\|^2_{L^2_x(L^2_v)}+\|\Tdx\Phi_0\|^2_{L^2_{x}}\).
\ee

For $I_3$, it holds that
\be
I_3=3\int_{1+|\xi|\ge \frac{r_0}{\eps}}e^{-2(1+|\xi|^2)t }|\hat{n}_0|^2 d\xi
\le  C e^{-\frac{r_0^2}{\eps^2}t }\|f_0\|^2_{L^2_x(L^2_{v})}. \label{S_6}
\ee

Therefore, it follows from \eqref{S_4}--\eqref{S_6} that
\be
\left\|e^{\frac{t}{\eps^2}B_\eps}f_0-e^{tD}n_0\sqrt{M}\right\|^2_{L^2_{x}(L^2_{v})}
\le C\(\eps^2e^{-t}+e^{-2\frac{at}{\eps^2}}\)\|f_0\|^2_{H^1_{P}}. \label{aaa}
\ee

Similarly, we  can obtain
\bma
&\left\|\Tdx\Delta^{-1}_x(e^{\frac{t}{\eps^2}B_\eps}f_0,\sqrt{M})-\Tdx\Delta^{-1}_xe^{tD}n_0\right\|^2_{L^2_x}\nnm\\
\le &3\int_{1+|\xi|\le \frac{r_0}{\eps}} \left|\frac{\xi}{|\xi|^2}(S_1(t,\xi,\eps)\hat{f}_0,\sqrt{M})-\frac{\xi}{|\xi|^2}e^{tD(\xi)}\hat{n}_0\right|^2d\xi\nnm\\
&+3\intr  \frac{1}{|\xi|^2}\left|(S_2(t,\xi,\eps)\hat{f}_0,\sqrt{M})\right|^2d\xi+3\int_{1+|\xi|\ge \frac{r_0}{\eps}} \frac{1}{|\xi|^2} |e^{tD(\xi)}\hat{n}_0|^2d\xi\nnm\\
\le & C\(\eps^2e^{-t}+e^{-2\frac{at}{\eps^2}}\)\|f_0\|^2_{H^1_{P}}. \label{bbb}
\ema
Combining \eqref{aaa} and \eqref{bbb}, we obtain \eqref{limit1}.

Then, we deal with \eqref{limit1a}.
Since $f_0\in N_0,$ by Lemma \ref{S2a} we have
\bma
I_2&\le C\int_{1+|\xi|\le \frac{r_0}{\eps}}\eps^2 (1+|\xi|^2)e^{-2\frac{at}{\eps^2}}\| \hat{f}_0\|^2_{\xi}d\xi+C\int_{1+|\xi|\ge \frac{r_0}{\eps}}e^{-2\frac{at}{\eps^2}}\| \hat{f}_0\|^2_{\xi}d\xi \nnm\\
&\le C\eps^2 e^{-2\frac{at}{\eps^2}}\(\|f_0\|^2_{H^1_x(L^2_v)}+\|\Tdx\Phi_0\|^2_{L^2_{x}}\).\label{S_7a}
\ema
Moreover, we can bound $I_3$ by
\be
I_3=3\int_{1+|\xi|\ge \frac{r_0}{\eps}}e^{-2(1+|\xi|^2)t }|\hat{n}_0|^2 d\xi
\le  C\eps^2 e^{-\frac{r_0^2}{\eps^2}t }\|f_0\|^2_{H^1_x(L^2_{v})}. \label{S_6a}
\ee
Thus, by \eqref{I1}, \eqref{S_7a} and \eqref{S_6a}  we obtain \eqref{limit1a}.
This proves the lemma.
\end{proof}

We have the second order fluid approximation of the semigroup $e^{\frac{t}{\eps^2}B_\eps}$ as follow.

\begin{lem}\label{fl2}
For any integer $k\ge 0$ and  any $f_0\in H^{k+1}_{P}$ satisfying $P_0f_0=0$,
\be
\bigg\|\frac1{\eps}e^{\frac{t}{\eps^2}B_\eps}f_0+e^{tD}\divx m_0\sqrt{M}\bigg\|_{H^k_{P}}
\le C\(\eps t^{-\frac12}e^{-\frac{t}{2}}+ \frac1{\eps}e^{-\frac{at}{\eps^2}}\)\|f_0\|_{H^{k+1}_{x}(L^2_v)}, \label{limit2}
\ee
where $m_0=(f_0,v\sqrt{M}),$ and $a>0$, $C>0$ are two constants independent of $\eps$.
\end{lem}

\begin{proof} For simplicity, we only prove the case of $k=0$.
By Theorem \ref{rate1} and taking $\eps\le r_0/2$ with $r_0>0$ given in Lemma \ref{eigen_4}, we have
\bma
&\left\|\frac1{\eps}e^{\frac{t}{\eps^2}B_\eps}f_0+e^{tD}\divx m_0\sqrt{M}\right\|^2_{L^2_{x}(L^2_{v})}\nnm\\
\le &3\int_{1+|\xi|\le \frac{r_0}{\eps}} \left\|\frac1{\eps}S_1(t,\xi,\eps)\hat{f}_0+e^{tD(\xi)}(\hat{m}_0\cdot\xi)\sqrt{M}\right\|^2d\xi\nnm\\
&+3\intr \left\|\frac1{\eps}S_2(t,\xi,\eps)\hat{f}_0\right\|^2d\xi+3\int_{1+|\xi|\ge \frac{r_0}{\eps}}  \left|e^{tD(\xi)}(\hat{m}_0\cdot\xi)\right|^2d\xi\nnm\\
=:&I_1+I_2+I_3. \label{S_7}
\ema
We estimate $I_j$, $j=1,2,3$ as follows. Since for any $f_0\in L^2_{x}(L^2_{v})$ satisfying $P_0f_0=0$,
$$
S_1(t,\xi,\eps)\hat{f}_0=e^{-(1+|\xi|^2)t+O(\eps^2(1+|\xi|^2)^2)t}\[-\eps(\hat{m}_0\cdot\xi)\sqrt{M}+O\(\eps^2(1+|\xi|^2)\)\],
$$
 it follows that
\bma
I_1=&3\int_{1+|\xi|\le \frac{r_0}{\eps}} \Big\|e^{-(1+|\xi|^2)t+O(\eps^2(1+|\xi|^2)^2)t}\[(\hat{m}_0\cdot\xi)\sqrt{M}-O\(\eps(1+|\xi|^2)\)\]\nnm\\
&\qquad-e^{-(1+|\xi|^2)t}(\hat{m}_0\cdot\xi)\sqrt{M}\Big\|^2d\xi\nnm\\
\le & C\eps^2\int_{1+|\xi|\le \frac{r_0}{\eps}}e^{-(1+|\xi|^2)t}\(  r_0^2(1+|\xi|^2)^3t^2|(\hat{m}_0\cdot\xi)|^2 +(1+|\xi|^2)^2\|\hat{f}_0\|^2 \)d\xi\nnm\\
\le& C\eps^2t^{-1} e^{-t}\| f_0\|^2_{H^1_{x}(L^2_{v})}.
\ema

By \eqref{S2} and $P_0f_0=0$, we have
\be
I_2 \le C\intr \frac1{\eps^2} e^{-2\frac{at}{\eps^2}}\|\hat{f}_0\|^2 d\xi
\le C\frac1{\eps^2}e^{-2\frac{at}{\eps^2}}\|f_0\|^2_{L^2_{x}(L^2_{v})}.
\ee

For $I_3$, it holds that
\be
I_3= 3\int_{1+|\xi|\ge \frac{r_0}{\eps}}  e^{-2(1+|\xi|^2)t }|(\hat{m}_0\cdot\xi)|^2d\xi
\le  C  e^{-\frac{r_0^2}{\eps^2}t }\|f_0\|^2_{H^1_{x}(L^2_{v})}. \label{S_9}
\ee

Therefore, it follows from \eqref{S_7}--\eqref{S_9} that
\be
\left\|\frac1{\eps}e^{\frac{t}{\eps^2}B_\eps}f_0+e^{tD}\divx m_0\sqrt{M}\right\|^2_{L^2_{x}(L^2_{v})}
\le
C\(\eps^2t^{-1} e^{-t}+ \frac1{\eps^2}e^{-2\frac{at}{\eps^2}}\)\|f_0\|^2_{H^1_{x}(L^2_v)}. \label{expan1}
\ee

Similarly, we have
\bma
&\left\|\frac1{\eps}\Tdx\Delta^{-1}_x(e^{\frac{t}{\eps^2}B_\eps}f_0,\sqrt{M})-\Tdx\Delta^{-1}_xe^{tD}\divx m_0\right\|^2_{L^2_x}\nnm\\
\le &3\int_{1+|\xi|\le \frac{r_0}{\eps}} \left|\frac{\xi}{\eps(1+|\xi|)^2}(S_1(t,\xi,\eps)\hat{f}_0,\sqrt{M})-\frac{\xi}{|\xi|^2}e^{tD(\xi)}(\hat{m}_0\cdot\xi)\right|^2d\xi\nnm\\
&+3\intr  \frac{1}{\eps^2|\xi|^2}\left| (S_2(t,\xi,\eps)\hat{f}_0,\sqrt{M})\right|^2d\xi+3\int_{1+|\xi|\ge \frac{r_0}{\eps}} \frac{1}{|\xi|^2}\left|e^{tD(\xi)}(\hat{m}_0\cdot\xi)\right|^2d\xi\nnm\\
\le &C\(\eps^2t^{-1} e^{-t}+ \frac1{\eps^2}e^{-2\frac{at}{\eps^2}}\)\|f_0\|^2_{H^1_{x}(L^2_v)}. \label{expan2}
\ema
Combining \eqref{expan1} and \eqref{expan2}, we obtain \eqref{limit2}.
This proves the lemma.
\end{proof}

\section{Diffusion limit}
\setcounter{equation}{0}
\label{sect4}
In this section, we study the diffusion limit of the solution to the nonlinear VPFP system \eqref{VPFP4}--\eqref{VPFP6} based on the fluid approximations of the solution to the linear VPFP system given in Section 3.

Since the operators $B_{\eps}$ and $A$ generate contraction semigroups in $H^k_P$ and $H^k_x$ $(k\ge 0)$ respectively, the solution $f_{\eps}(t)$ to the VPFP system \eqref{VPFP4}--\eqref{VPFP6} and the solution $n(t)$ to the DDP system \eqref{n_1}--\eqref{n_2} can be represented by
\bma
f_{\eps}(t)&=e^{\frac{t}{\eps^2}B_{\eps}}f_0+\intt \frac1{\eps}e^{\frac{t-s}{\eps^2}B_{\eps}}G(f_{\eps})(s)ds, \label{fe}
\\
 n(t)&=e^{tD}n_0-\intt e^{(t-s)D}\divx(n\Tdx\Phi )(s)ds, \label{ne}
\ema
where
$$n_0=\intr f_0\sqrt{M}dv.$$


By taking inner product between $\sqrt M$ and \eqref{VPFP4}, we obtain
\be
\dt n_{\eps}+\frac{1}{\eps}\divx m_{\eps}=0, \label{G_3a}
\ee
where
$$n_{\eps}=(f_{\eps},\sqrt M), \quad m_{\eps}=(f_{\eps},v\sqrt M).$$
Taking the microscopic projection $ P_1$ to \eqref{VPFP4}, we have
\bma
&\dt( P_1f_{\eps})+ \frac{1}{\eps}P_1(v\cdot\Tdx  P_1f_{\eps})-\frac{1}{\eps}v\sqrt M\cdot\Tdx\Phi_{\eps}-\frac{1}{\eps^2}L( P_1f_{\eps})\nnm\\
=&- \frac{1}{\eps}P_1(v\cdot\Tdx  P_0f_{\eps})+ \frac{1}{\eps}P_1 G(f_{\eps}).\label{GG2}
\ema
By \eqref{GG2}, we can express the microscopic part $ P_1f_{\eps}$ as
\bq   P_1f_{\eps}=  L^{-1}[\eps^2\dt(P_1f_{\eps})+ \eps P_1(v\cdot\Tdx  P_1f_{\eps})- \eps P_1 G ]-\eps v\sqrt{M}\cdot (\Tdx n_{\eps}-\Tdx \Phi_{\eps}). \label{p_c}\eq
Substituting \eqref{p_c} into \eqref{G_3a}, we obtain
\be
\dt n_{\eps}+ \eps\dt\divx m_{\eps} + n_{\eps}- \Delta_x n_{\eps}=-\divx(v\cdot\Tdx P_1f_{\eps}-P_1G ,v\sqrt M). \label{G_9a}
\ee
Define the energy  $E(f_{\eps})$ and the dissipation $D(f_{\eps})$ by
\bmas
 E(f_{\eps})&= \|f_{\eps}(t)\|^2_{H^4_x(L^2_v)}+\|\Tdx\Phi_{\eps}(t)\|^2_{H^4_x},\\
D(f_{\eps})&= \frac{1}{\eps^2}\|P_1f_{\eps}\|^2_{H^4_x(L^2_\sigma)}+ \| P_0f_{\eps}\|^2_{H^4_x(L^2_v)}  +\| \Tdx\Phi_{\eps}\|^2_{H^4_x},
\emas
where $\Tdx\Phi_{\eps}=\Tdx \Delta^{-1}_x(f_{\eps},\sqrt M)$.

First, we establish the following energy estimate for the solution $f_{\eps}$ to the system \eqref{VPFP4}--\eqref{VPFP6}.

\begin{lem}\label{energy1}For any $\eps\ll 1$, there exists a small constant $\delta_0>0$ such that if $\|f_{0}\|_{H^4_x(L^2_v)}+\|(f_0,\sqrt{M})\|_{L^1_x}\le \delta_0$, then
the  system \eqref{VPFP4}--\eqref{VPFP6} admits a unique global solution $f_{\eps}(t)= f_{\eps}(t,x,v) $ satisfying the following energy estimate:
\be
  E(f_{\eps}(t))+ \intt  D(f_{\eps}(s))ds \le C\delta_0^2, \label{G_1}
\ee
where $C>0$ is a constant independent of $\eps$.  Moreover,  $f_{\eps}(t)$ has the following time-decay rate:
\be
  E(f_{\eps}(t)) \le C\delta_0^2 e^{-t}. \label{G_11}
\ee
\end{lem}

\begin{proof}

First, we establish the macroscopic energy estimate of $f_{\eps}$. Taking the inner product between $\dxa n_{\eps}$ and $\dxa\eqref{G_9a}$ with $|\alpha|\le 3$, we have
\bma
&\Dt \|\dxa n_{\eps}\|^2_{L^2_x}+2\Dt \eps\intr\dxa\divx m_{\eps}\dxa n_{\eps} dx+\frac32\(\|\dxa n_{\eps}\|^2_{L^2_x}+\|\dxa \Tdx n_{\eps}\|^2_{L^2_x}\)\nnm\\
\le& C\|\dxa\Tdx P_1f_{\eps}\|^2_{L^2_{x}(L^2_{v})}+CE(f_{\eps})D(f_{\eps}).\label{en_5}
\ema
Similarly, taking the inner product between $-\dxa \Phi_{\eps}$ and $\dxa\eqref{G_9a}$ with $|\alpha|\le 3$ we obtain
\bma
&\Dt \|\dxa \Tdx\Phi_{\eps}\|^2_{L^2_x}+2\Dt \eps\intr\dxa m_{\eps}\dxa \Tdx\Phi_{\eps} dx+\frac32(\|\dxa \Tdx\Phi_{\eps}\|^2_{L^2_x}+\|\dxa  n_{\eps}\|^2_{L^2_x})\nnm\\
\le& C\(\|\dxa P_1f_{\eps}\|^2_{L^2_{x}(L^2_{v})}+\|\dxa\Tdx P_1f_{\eps}\|^2_{L^2_{x}(L^2_{v})}\)+CE(f_{\eps})D(f_{\eps}).\label{en_6}
\ema
Taking the summation of $\eqref{en_5}+\eqref{en_6}$ with $|\alpha|\le 3$, we obtain
\bma
&\Dt \(\|n_{\eps}\|^2_{H^3_x} +\|\Tdx\Phi_{\eps}\|^2_{H^3_x}\)+\frac32\(\|\Tdx n_{\eps}\|^2_{H^3_x}+2\| n_{\eps}\|^2_{H^3_x}  +\| \Tdx\Phi_{\eps}\|^2_{H^3_x}\)\nnm\\
&+\Dt \sum_{ |\alpha|\le 3}\eps\(2 \intr\dxa\divx m_{\eps}\dxa n_{\eps} dx+2 \intr\dxa m_{\eps}\dxa \Tdx\Phi_{\eps} dx\)\nnm\\
\le & CE(f_{\eps})D(f_{\eps})+C\| P_1f_{\eps}\|^2_{H^4_x(L^2_{v})}.\label{E_1a}
\ema

Next, we deal with the microscopic energy estimate of $f_{\eps}$. By taking the inner product between $\dxa  f_{\eps}$ and $\dxa\eqref{VPFP4}$ with $|\alpha|\le 4$, we have
\bma
&\frac12\Dt \(\|\dxa f_{\eps}\|^2_{L^2_{x}(L^2_{v})}+\|\dxa\Tdx\Phi_{\eps}\|^2_{L^2_x}\)-\frac{1}{\eps^2}\intr (L\dxa f_{\eps})\dxa f_{\eps}dxdv
\nnm\\
=&\frac{1}{2\eps}\intrr  \dxa (v\cdot\Tdx \Phi_{\eps} f_{\eps})\dxa P_1f_{\eps}dxdv-\frac{1}{\eps}\intrr \dxa (\Tdx \Phi_{\eps}\cdot \Tdv f_{\eps})\dxa P_1f_{\eps}dxdv\nnm\\
\le &\frac{\mu}{2\eps^2}\|  \dxa P_1f_{\eps}\|^2_{L^2_{x}(L^2_{\sigma})}+ CE(f_{\eps})D(f_{\eps}),\label{G_0}
\ema
which leads to
\bma
 \Dt \(\|f_{\eps}\|^2_{H^4_x(L^2_{v})}+\|\Tdx\Phi_{\eps}\|^2_{H^4_x}\)+\frac{\mu}{\eps^2}\| P_1f_{\eps}\|^2_{H^4_x(L^2_{\sigma})}
\le CE(f_{\eps})D(f_{\eps}).\label{G_00}
\ema
Taking the summation of $C_1\eqref{E_1a}+\eqref{G_00}$ with $C_1>0 $ large and $\eps>0$ small  to get
\bma
 &\Dt \(\|f_{\eps}\|^2_{H^4_x(L^2_{v})}+C_1\|n_{\eps}\|^2_{H^3_x}+(C_1+1)\|\Tdx\Phi_{\eps}\|^2_{H^4_x}\)\nnm\\
 &+\Dt C_1\sum_{ |\alpha|\le 3}\eps\(2 \intr\dxa\divx m_{\eps}\dxa n_{\eps} dx+2 \intr\dxa m_{\eps}\dxa \Tdx\Phi_{\eps} dx\)\nnm\\
  &+  \frac{\mu}{\eps^2}\|P_1f_{\eps}\|^2_{H^4_x(L^2_\sigma)}+\frac32C_1\(\|\Tdx n_{\eps}\|^2_{H^3_x}+2\| n_{\eps}\|^2_{H^3_x}  +\| \Tdx\Phi_{\eps}\|^2_{H^3_x}\)  \nnm\\
 \le& CE(f_{\eps})D(f_{\eps}) . \label{G_10}
\ema

Assume that $E(f_{\eps})\le C\delta_0$. It follows from \eqref{G_10} that for $\eps>0$ and $\delta_0>0$ sufficiently small, there exist two functionals $E_1(f_\eps)\sim E(f_\eps)$ and $D_1(f_\eps)\sim D(f_\eps)$ such that 
\be
\Dt E_1(f_\eps)+ D_1(f_\eps)\le 0,\quad D_1(f_\eps)\ge E_1(f_\eps). \label{G_2}
\ee

Since
\bmas
\|\Tdx\Phi_0\|^2_{H^4_x}&=\intr \frac1{|\xi|^2}(1+|\xi|^2)^2|( \hat{f}_0,\sqrt{M})|^2d\xi \\
&\le C\sup_{|\xi|\le1}|( \hat{f}_0,\sqrt{M})|^2\int_{|\xi|\le1} \frac1{|\xi|^2}d\xi+\int_{|\xi|\ge 1} (1+|\xi|^2)^2|( \hat{f}_0,\sqrt{M})|^2d\xi \\
&\le C\(\|f_0\|^2_{H^4_x(L^2_v)}+\|( f_0,\sqrt{M})\|^2_{L^1_x}\),
\emas
we can prove \eqref{G_1} and \eqref{G_11} by \eqref{G_2} for $\eps>0$ and $\delta_0>0$ sufficiently small.
This completes the proof of the lemma.
\end{proof}

\begin{lem}\label{energy2}For any $\eps\ll 1$, there exists a small constant $\delta_0>0$ such that if $\|f_{0}\|_{H^4_x(L^2_v)}+\|\Tdv f_{0}\|_{H^3_x(L^2_v)}+\||v| f_{0}\|_{H^3_x(L^2_v)}+\|(f_0,\sqrt{M})\|_{L^1_x}\le \delta_0$, then
the solution $f_{\eps}(t)= f_{\eps}(t,x,v) $ to the system \eqref{VPFP4}--\eqref{VPFP6} satisfies
\bma
 \|f_{\eps}\|^2_{H^4_x(L^2_v)}+\|\Tdv f_{\eps}\|^2_{H^3_x(L^2_v)}+\||v| f_{\eps}\|^2_{H^3_x(L^2_v)}+\|\Tdx\Phi_{\eps}\|^2_{H^4_x}
 \le C\delta^2_0e^{-t}  , \label{G_10a}
\ema
where $C>0$ is a constant independent of  $\eps$.
\end{lem}

\begin{proof}
By taking $P_1$ to \eqref{VPFP4} and noting that $P_1Lf_{\eps}=LP_1f_{\eps}$, we have
\bma
&\dt P_1 f_{\eps}+\frac{1}{\eps}v\cdot\Tdx P_1f_{\eps}-\frac{1}{\eps^2}\Delta_v P_1f_{\eps}+\frac{1}{\eps^2}\frac{|v|^2}4P_1 f_{\eps}-\frac{1}{\eps^2}\frac32 P_1 f_{\eps}\nnm\\
=&\frac{1}{\eps}v\sqrt{M}\cdot\Tdx \Phi_{\eps}-\frac{1}{\eps}v\cdot\Tdx P_0 f_{\eps}+\frac{1}{\eps}P_0(v\cdot\Tdx P_1 f_{\eps})\nnm\\
&+\frac1{2\eps}v\cdot\Tdx\Phi_{\eps} f_{\eps} + \frac{1}{\eps}\Tdx\Phi_{\eps}\cdot \Tdv f_{\eps}.\label{b_1a}
\ema
Taking the inner product between $|v|^2\dxa P_1 f_{\eps}$ and $\dxa\eqref{b_1a}$ with $|\alpha|\le 3$,  we obtain
\bma
&\Dt \||v| P_1f_{\eps}\|^2_{H^3_x(L^2_{v})}+ \frac{1}{\eps^2}\(\||v|\Tdv P_1 f_{\eps}\|^2_{H^3_x(L^2_{v})}+\frac14\||v|^2 P_1f_{\eps}\|^2_{H^3_x(L^2_{v})}\)\nnm\\
\le&  C\(\|\Tdx f_{\eps}\|^2_{H^3_x(L^2_{v})}+\| \Tdx\Phi_{\eps}\|^2_{H^3_x}\)+\frac{C}{\eps^2}\| |v|P_1f_{\eps}\|^2_{H^3_x(L^2_{v})}\nnm\\
&  +C\|\Tdx\Phi_{\eps}\|^2_{H^{3}_x}\(\||v| f_{\eps}\|^2_{H^{3}_x(L^2_v)}+\|\Tdv f_{\eps}\|^2_{H^{3}_x(L^2_v)}\).\label{aa}
\ema

Taking $\dvb$ with $|\beta|=1$ to \eqref{b_1a}, we get
\bma
&\dt \dvb P_1f_{\eps}+\frac{1}{\eps}v\cdot\Tdx \dvb P_1f_{\eps}-\frac{1}{\eps^2}\Delta_v \dvb P_1f_{\eps}+\frac{1}{\eps^2}\frac{|v|^2}2\dvb P_1f_{\eps}-\frac{1}{\eps^2}\frac32\dvb P_1f_{\eps}\nnm\\
=&\frac{1}{\eps}\dxb P_1f_{\eps}-\frac{1}{\eps^2}v^\beta P_1f_{\eps}+\frac{1}{\eps}\dvb(v\sqrt{M})\cdot\Tdx\Phi_{\eps}-\frac{1}{\eps}\dvb (v\cdot\Tdx P_0 f_{\eps})+\frac{1}{\eps}\dvb P_0(v\cdot\Tdx P_1 f_{\eps})\nnm\\
&+\frac1{2\eps}v\cdot\Tdx\Phi_{\eps} \dvb f_{\eps}+ \frac1{2\eps}\dxb\Phi_{\eps} f_{\eps} + \frac{1}{\eps}\Tdx\Phi_{\eps}\cdot \Tdv\dvb f_{\eps}.\label{b_1}
\ema
Taking the inner product between $\dxa\dvb P_1f_{\eps}$ and $\dxa\eqref{b_1}$ with $|\alpha|\le 3$,  we obtain
\bma
&\Dt \|\Tdv P_1f_{\eps}\|^2_{H^3_x(L^2_{v})}+ \frac{1}{\eps^2}\(\| \Tdv^2 P_1f_{\eps}\|^2_{H^3_x(L^2_{v})}+\||v|\Tdv P_1f_{\eps}\|^2_{H^3_x(L^2_{v})}\)\nnm\\
\le& C \(\| \Tdx f_{\eps}\|^2_{H^3_x(L^2_{v})}+\|\Tdx\Phi_{\eps}\|^2_{H^3_x}\)+\frac{C}{\eps^2}\(\| \Tdv P_1f_{\eps}\|^2_{H^3_x(L^2_{v})}+\||v| P_1f_{\eps}\|^2_{H^3_x(L^2_{v})}\)\nnm\\
& + C\|\Tdx\Phi_{\eps}\|^2_{H^{3}_x}\(\|f_{\eps}\|^2_{H^{3}_x(L^2_v)}+\|\Tdv f_{\eps}\|^2_{H^{3}_x(L^2_v)}\).\label{aa1}
\ema

Taking the summation of $C_2 \eqref{G_10}+\eqref{aa}+\eqref{aa1}$  with $C_2>0$ large and taking  $\delta_0>0$ sufficiently small, we can prove \eqref{G_10a} by a similar argument as \eqref{G_10}.
\end{proof}

By a similar argument as Lemma \ref{energy1}, we can show

\begin{lem}\label{energy3}
There exists a small constant $\delta_0>0$ such that if $\|n_0\|_{H^4_x}+\|n_{0}\|_{L^1_x}\le \delta_0$, then the  system \eqref{n_1}--\eqref{n_2} admits a unique global solution $n(t)$ satisfying
\be
\|n (t)\|_{H^4_x}+\|\Tdx\Phi(t)\|_{H^4_x}\le C\delta_0 e^{-\frac t2}, \label{DDP}
\ee
where $\Tdx\Phi(t)=\Tdx\Delta^{-1}_xn(t)$ and $C>0$ is a constant.
\end{lem}

Theorem \ref{existence} is directly follows from Lemmas \ref{energy1}.

With the help of Lemmas \ref{energy1}--\eqref{energy3}, we can prove Theorem  \ref{thm1.1} as follows.


\begin{proof}[\underline{\textbf{Proof of Theorem \ref{thm1.1}}}]
Define
\be
Q_{\eps}(t)=\sup_{0\le s\le t}\(\eps e^{-\frac{s}{2}}+e^{-\frac{as}{\eps^2}}\)^{-1} \|f_{\eps}(s)-n(s)\sqrt M \|_{ H^{2}_P}, \label{Qt}
\ee
where the norm $\|\cdot\|_{H^2_P}$ is defined by \eqref{H2P}.

We claim that
\be
Q_{\eps}(t) \le C\delta_0 ,\quad \forall t>0. \label{limit6}
\ee
It is easy to verify that the estimate  \eqref{limit0}  follows  from \eqref{limit6}.

 By \eqref{fe} and \eqref{ne}, we have
\bma
\|f_{\eps}(t)-n(t)\sqrt{M}\|_{ H^2_P}&\le \left\|e^{\frac{t}{\eps^2}B_{\eps}}f_0-e^{tD}n_0\sqrt{M}\right\|_{ H^2_P}\nnm\\
&\quad+\intt \bigg\|\frac1{\eps}e^{\frac{t-s}{\eps^2}B_{\eps}}G(f_{\eps})-e^{(t-s)D}\divx(n\Tdx\Phi )\sqrt{M}\bigg\|_{ H^2_P}ds\nnm\\
&=:I_1+I_2. \label{f2}
\ema

By \eqref{limit1}, we can bound $I_1$ by
\be I_1\le C\delta_0\(\eps e^{-\frac{t}{2}}+e^{-\frac{at}{\eps^2}}\).\ee

To estimate $I_2$, we decompose
\bma
I_2&\le \intt \bigg\|\frac1{\eps}e^{\frac{t-s}{\eps^2}B_{\eps}}G(f_{\eps})-e^{(t-s)D}\divx (n_{\eps}\Tdx\Phi_{\eps} )\sqrt{M}\bigg\|_{ H^2_P}ds \nnm\\
&\quad+\intt\left\|e^{(t-s)D}\[\divx (n_{\eps}\Tdx\Phi_{\eps} )-\divx(n\Tdx\Phi )\]\sqrt{M}\right\|_{ H^2_P}ds \nnm\\
&=:I_3+I_4. \label{l0}
\ema

Since $(G(f_{\eps}),\sqrt{M})=0$ and $(G(f_{\eps}),v\sqrt{M})=n_{\eps}\Tdx\Phi_{\eps},$ it follows from \eqref{limit2} that
\bma
I_3\le &C \intt\(\eps (t-s)^{-\frac{1}{2}} e^{-\frac{t-s}{2}}+ \frac1{\eps}e^{-\frac{a(t-s)}{\eps^2}}\)\|G(f_{\eps})\|_{H^3_x(L^2_v)}ds\nnm\\
\le &C \delta_0^2\intt \(\eps (t-s)^{-\frac{1}{2}}e^{-\frac{t-s}{2}}+ \frac1{\eps}e^{-\frac{a(t-s)}{\eps^2}}\)e^{-s}ds, \label{l1}
\ema
where we have used (refer to \eqref{G_10a})
$$\|G(f_{\eps})\|_{H^3_x(L^2_v)}\le C\(\|| v| f_{\eps} \|_{H^3_x(L^2_v)} +\| \Tdv f_{\eps}\|_{H^3_x(L^2_v)}\)\|\Tdx\Phi_{\eps}\|_{H^3_x}\le C\delta_0^2e^{-t}.$$
Denote
\be
J_1=\intt \frac1{\eps}e^{-\frac{a(t-s)}{\eps^2}}e^{-s}ds. \label{e2}
\ee
It holds that
\bma
J_1&=\(\int^{t/2}_0+\int^{t}_{t/2}\) \frac1{\eps}e^{-\frac{a(t-s)}{\eps^2}}e^{-s} ds \nnm\\
&\le  \int^{t/2}_0\frac1{\eps}e^{-\frac{a(t-s)}{\eps^2}} ds+e^{-\frac t2}\int^{t}_{t/2} \frac1{\eps}e^{-\frac{a(t-s)}{\eps^2}} ds \nnm\\
&\le C\(\eps e^{-\frac{at}{2\eps^2}}+\eps e^{-\frac t2}\). \label{e1}
\ema
Thus, it follows from \eqref{l1}--\eqref{e1} that
\be
I_3 \le C \delta_0^2\eps \(e^{-\frac{t}{2}}+e^{-\frac{at}{2\eps^2}}\). \label{l3}
\ee

For any vector $F=(F_1,F_2,F_3)\in H^k_x$, we have
$$
\|e^{tD}\divx F\sqrt{M}\|^2_{H^k_P}\le C\intr e^{-2(1+|\xi|^2)t}(1+|\xi|^2)^{k+1} |\hat{F} |^2d\xi
\le Ct^{-1}e^{-t}\|F\|^2_{H^k_x}.
$$
This together with \eqref{Qt}, \eqref{DDP} and \eqref{G_10a} imply that
\bma
I_4\le &C\intt (t-s)^{-\frac12}e^{-\frac{t-s}2}\Big(\|n_{\eps}-n\|_{H^2_x}\|\Tdx\Phi_{\eps}\|_{H^2_x}\nnm\\
&\qquad+\| n\|_{H^2_x}\|\Tdx\Phi_{\eps}-\Tdx\Phi\|_{H^2_x}\Big)ds \nnm\\
\le &C \delta_0 Q_{\eps}(t)\intt (t-s)^{-\frac12}e^{-\frac{t-s}2}\(\eps e^{-s}+ e^{-(\frac12+\frac{a}{\eps^2})s}\)ds. \label{l2}
\ema
Denote
\be
J_2=\intt (t-s)^{-\frac12}e^{-\frac{t-s}2} e^{-(\frac12+\frac{a}{\eps^2})s} ds.
\ee
Since
$$\frac{1}{\eps}\sqrt{s}e^{-\frac{a}{\eps^2}s}\le C e^{-\frac{a}{2\eps^2}s},$$
we have
\be
J_2\le C\eps e^{-\frac t2}\(\int^{t/2}_0+\int^{t}_{t/2}\) (t-s)^{-\frac12}s^{-\frac12}  ds
\le C\eps e^{-\frac t2}. \label{e3}
\ee
Thus, it follows from \eqref{l2}--\eqref{e3} that
\be
I_4 \le C \delta_0 Q_{\eps}(t)\eps  e^{-\frac{t}{2}} . \label{l4}
\ee

By combining \eqref{f2}--\eqref{l0}, \eqref{l3}  and \eqref{l4}, we obtain
\be
Q_{\eps}(t)\le  C\delta_0+C\delta_0^2+C\delta_0 Q_{\eps}(t),
\ee
where $C>0$ is a constant independent of $\eps$. By taking $\delta_0>0$  small enough, we can prove \eqref{limit6}. By \eqref{limit1a} and a similar argument as above, we can prove \eqref{limit_1a}.
\end{proof}

\section{Appendix}
\setcounter{equation}{0}
\label{Appendix}
We give the proof of Theorem \ref{rate1} in the following.

\begin{proof}[\underline{\textbf{Proof of Theorem \ref{rate1}}}]
For simplicity, we assume that $0<\eps\le r_0/2$. By Theorem 2.7 in \cite{Pazy}, it is sufficient to prove \eqref{B_0} for $f\in D(B_{\eps}(\xi)^2)$ because the domain $D(B_{\eps}(\xi)^2)$ is dense in $L^2_{\xi}(\R^3_v)$. By Corollary 7.5 in \cite{Pazy}, the semigroup $e^{\frac{t}{\eps^2}B_{\eps}(\xi)}$ can be represented by
 \bq
  e^{\frac{t}{\eps^2}B_{\eps}(\xi)}f =\frac1{2\pi i}\int^{\kappa+ i\infty}_{\kappa- i\infty}
   e^{ \frac{\lambda t}{\eps^2}}(\lambda-B_{\eps}(\xi))^{-1}fd\lambda,  \quad  f\in D(B_{\eps}(\xi)^2),\,\, \kappa>0.          \label{V_3}
 \eq

First of all, we investigate the formula \eqref{V_3} for $\eps(1+|\xi|)\le r_0$. By \eqref{S_8} we have
 \bq
  (\lambda-B_{\eps}(\xi))^{-1} =[\lambda^{-1}  P_0 +(\lambda  P_1-Q_{\eps}(\xi))^{-1} P_1]-Z_{\eps}(\lambda,\xi)\label{V_1},
 \eq
with the operator $Z_{\eps}(\lambda,\xi)$ defined by
 \bma
 Z_{\eps}(\lambda,\xi)&=[\lambda^{-1}  P_0 +(\lambda  P_1-Q_{\eps}(\xi))^{-1} P_1]
   [I+Y_{\eps}(\lambda,\xi)]^{-1} Y_{\eps}(\lambda,\xi),  \label{V_1a}
\\
 Y_{\eps}(\lambda,\xi)&=  i\eps \lambda^{-1}  P_1(v\cdot\xi)(1+\frac1{|\xi|^2})  P_0
    + i\eps P_0(v\cdot\xi) P_1(\lambda  P_1-Q_{\eps}(\xi))^{-1} P_1. \label{V_1b}
 \ema

By a similar argument as Lemma 3.1 in \cite{Li3}, we conclude from \eqref{S_3} that the operator $Q_{\eps}(\xi) $ generates a strongly continuous contraction
semigroup on $N_0^\bot$, which satisfies for any $t>0$ and $f\in N_0^\bot $ that
  \bq
    \|e^{tQ_{\eps}(\xi)}f\|\leq e^{- t}\|f\|. \label{decay_1}
 \eq
In addition, for any $x>-1 $ and $f\in N_0^\bot $, it holds that
\bq
 \int^{+\infty}_{-\infty}\|[(x+ i y) P_1-Q_{\eps}(\xi)]^{-1}f\|^2dy \leq \pi(x+1 )^{-1}\|f\|^2.\label{S_4}
\eq
Substituting \eqref{V_1} into \eqref{V_3}, we have the following decomposition of the semigroup $e^{\frac{t}{\eps^2}B_{\eps}(\xi)}$
 \be
 e^{\frac{t}{\eps^2}B_{\eps}(\xi)}f = P_0f+e^{\frac{t}{\eps^2}Q_{\eps}(\xi)} P_1f -\frac1{2\pi i}\int^{\kappa+ i\infty}_{\kappa- i\infty}
    e^{\frac{\lambda t}{\eps^2}}Z_{\eps}(\lambda,\xi)fd\lambda,  \quad \eps(1+|\xi|)\le r_0.  \label{V_3a}
  \ee

To estimate the last term on the right hand side of \eqref{V_3a}, let us denote
 \bq
 U_{\kappa,N}=\frac1{2\pi i}\int^{N}_{-N} e^{\frac{t}{\eps^2}(\kappa+ i y)}Z_{\eps}(\kappa+iy,\xi)f dy,    \label{UsN}
 \eq
where the constant $N>0$ is chosen large enough so that $N>y_0$ with $y_0$ defined in Lemma \ref{spectrum2}. Since $Z_{\eps}(\lambda,\xi)$ is analytic on the domain ${\rm Re}\lambda>-1/2$ with only finite singularities at 
$\lambda=\lambda_0(|\xi|,\eps)\in \sigma(B_{\eps}(\xi))$ and $\lambda=0$, we can shift the integration
\eqref{UsN} from the line ${\rm Re}\lambda=\kappa>0$ to ${\rm Re}\lambda=-1/2$. Then by  the Residue Theorem, we obtain
 \bma
 U_{\kappa,N} =& {\rm Res} \lt\{e^{\frac{\lambda t}{\eps^2}}Z_{\eps}(\lambda,\xi)f;\lambda_0(|\xi|,\eps)\rt\}+{\rm Res}
  \lt\{e^{\frac{\lambda t}{\eps^2}}Z_{\eps}(\lambda,\xi)f;0\rt\} \nnm\\
  &+U_{-\frac12,N}+J_N,    \label{UsN2}
 \ema
where Res$\{f(\lambda);\lambda_j\}$ means the residue of $f(\lambda)$ at $\lambda=\lambda_j$ and
$$
  J_N=\frac1{2\pi i}\(\int^{\kappa+ i N}_{-\frac12+ i N}
      -\int^{\kappa- i N}_{-\frac12- i N}\) e^{\frac{\lambda t}{\eps^2}}Z_{\eps}(\lambda,\xi)f d\lambda.
$$
The right hand side of \eqref{UsN2} is estimated as follows.
By Lemma \ref{LP}, it is easy to verify that
 \be
 \|J_N\|_{\xi}\rightarrow0, \quad\mbox{as}\quad N\rightarrow\infty.  \label{UsN2a}
 \ee
Let
 \be
 \lim_{N\to\infty} U_{-\frac12,N}(t) =U_{-\frac12,\infty}(t)
 =:\int^{-\frac12+i \infty}_{-\frac12-i \infty} e^{\frac{\lambda t}{\eps^2}}Z_{\eps}(\lambda,\xi)fd\lambda.   \label{UsN3}
\ee

Since it follows from Lemma \ref{LP} that $\|Y_{\eps}(-\frac12+ i y,\xi)\|_{\xi}\leq 1/2$ for $y\in \R$ and $\eps(1+|\xi|)\leq r_0$ with $r_0>0$ being sufficiently small, the operator $I-Y_{\eps}(-\frac12+ i y,\xi)$ is invertible on $L^2_{\xi}(\R^3_v)$ and satisfies
$\|[I-Y_{\eps}(-\frac12+ i y,\xi)]^{-1}\|_{\xi}\le 2$ for $y\in \R$ and $\eps(1+|\xi|)\leq r_0$.
Thus, we have for any $f,g\in L^2_{\xi}(\R^3_v)$
  \bmas
 &|(U_{-\frac12,\infty}(t)f,g)_{\xi}| \le e^{-\frac{t}{2\eps^2}}\int^{+\infty}_{-\infty} |(Z_{\eps}(\lambda,\xi)f,g)_{\xi}|dy
  \nnm\\
 &\leq C \eps(1+|\xi|)e^{-\frac{t}{2\eps^2}}\int^{+\infty}_{-\infty} \( \|[\lambda P_1-Q_{\eps}(\xi)]^{-1} P_1f\| +|\lambda|^{-1}\| P_0f\|_{\xi}\)
    \nnm \\
&\quad\quad\quad\quad
   \times\(\|[\overline{\lambda} P_1-Q_{\eps}(-\xi)]^{-1} P_1g\| +|\overline{\lambda}|^{-1}\| P_0g\|_{\xi}\) dy,\quad \lambda=-\frac12+ i y.
 \emas
This together with \eqref{S_4} and
$$
  \int^{+\infty}_{-\infty}\|(x+ i y)^{-1} P_0f\|^2_{\xi} dy = \pi|x|^{-1}\|P_0f\|^2_{\xi}
$$
imply that
 $
 |(U_{-\frac12,\infty}(t)f,g)_{\xi}| \leq Cr_0\  e^{-\frac{ t}{2\eps^2}}\|f\|_{\xi}\|g\|_{\xi},
 $
and
 \bq
\|U_{-\frac12,\infty}(t)\|_{\xi} \le Cr_0 e^{-\frac{ t}{2\eps^2}}. \label{UsN4}
 \eq

Since $\lambda_0(|\xi|,\eps),0\in \rho(Q_{\eps}(\xi))$, and
$$Z_{\eps}(\lambda,\xi)=\lambda^{-1}  P_0 +(\lambda P_1-Q_{\eps}(\xi))^{-1} P_1-(\lambda-B_{\eps}(\xi))^{-1},$$
we can obtain
 \bma
 {\rm Res}\{e^{\lambda t}Z_{\eps}(\lambda,\xi)f;\lambda_0(|\xi|,\eps)\}
  &=-{\rm Res}\{e^{\lambda t}(\lambda-B_{\eps}(\xi))^{-1}f;\lambda_0(|\xi|,\eps)\}\nnm\\
 & =-e^{\lambda_0(|\xi|,\eps)t}\Big(f,\overline{\psi_0(\xi,\eps)}\Big)_{\xi}\psi_0(\xi,\eps), \label{projection}\\
  {\rm Res}\{e^{\lambda t}Z_{\eps}(\lambda,\xi)f;0\}
  &={\rm Res}\{e^{\lambda t}\lambda^{-1}P_0f;0\}= P_0f. \label{projection1}
\ema

Therefore, we conclude from \eqref{V_3a} and \eqref{UsN}--\eqref{projection1} that
 \be
   e^{\frac{t}{\eps^2}B_{\eps}(\xi)}f= e^{\frac{t}{\eps^2}Q_{\eps}(\xi)} P_1f+U_{-\frac12,\infty}(t)
  + e^{\frac{t}{\eps^2}\lambda_0(|\xi|,\eps)} \Big(f,\overline{\psi_0(\xi,\eps)}\Big)_{\xi}\psi_0(\xi,\eps) , \label{low}
 \ee
for $\eps(1+|\xi|)\leq r_0$.

 Next, we turn to investigate the formula \eqref{V_3} for $\eps(1+|\xi|)> r_0$. It holds that $\eps|\xi|> r_0/2$.
By \eqref{E_6} we have
 \bq
(\lambda-B_{\eps}(\xi))^{-1}=(\lambda-A_{\eps}(\xi))^{-1}+H_{\eps}(\lambda,\xi),\label{V_2}
 \eq
with
 \bma
 H_{\eps}(\lambda,\xi)&=(\lambda-A_{\eps}(\xi))^{-1}[I-G_{\eps}(\lambda,\xi)]^{-1}G_{\eps}(\lambda,\xi),\label{V_2a}\\
 G_{\eps}(\lambda,\xi)&=(K- i\eps(v\cdot\xi)|\xi|^{-2}P_0)(\lambda-A_{\eps}(\xi))^{-1}.\label{V_2b}
 \ema

By Lemma \ref{SG_1} and a similar argument as Lemma 3.1 in \cite{Li3}, we conclude that the operator $A_{\eps}(\xi) $ generates a strongly continuous contraction
semigroup on $L^2(\R^3_v)$, which satisfies for any $t>0$ and $f\in L^2(\R^3_v) $ that
  \bq
    \|e^{tA_{\eps}(\xi)}f\|\leq e^{- \nu_0t}\|f\|. \label{decay_2}
 \eq
In addition, for any $x>-\nu_0 $ and $f\in L^2(\R^3_v)$, it holds that
 \bq
\int^{+\infty}_{-\infty}\|(x+ i y-A_{\eps}(\xi))^{-1}f\|^2dy\leq \pi(x+\nu_0)^{-1}\|f\|^2. \label{VsNa5}
 \eq
 Substituting \eqref{V_2} into \eqref{V_3} yields
 \bq
 e^{\frac{t}{\eps^2}B_{\eps}(\xi)}f = e^{\frac{t}{\eps^2}A_{\eps}(\xi)}f +\frac1{2\pi i}\int^{\kappa+ i\infty}_{\kappa- i\infty}
    e^{\frac{\lambda t}{\eps^2}}H_{\eps}(\lambda,\xi)fd\lambda,\quad \eps(1+|\xi|)> r_0.   \label{V_3b}
 \eq

Similarly, in order to estimate the last term on the right hand side of \eqref{V_3b},
let us denote
 \bq
V_{\kappa,N} =\frac1{2\pi i}\int^{N}_{-N}e^{\frac{t}{\eps^2}(\kappa+ i y)}H_{\eps}(\kappa+ i y,\xi) dy    \label{VsN}
 \eq
for  sufficiently large constant $N>0$ as  in \eqref{UsN}.
Since the operator $H_{\eps}(\lambda,\xi)$ is analytic on the domain ${\rm Re}\lambda\ge -\sigma_0$ for the constant $\sigma_0= \alpha(r_0/2)/2$ with $\alpha(r)>0$ given by Lemma~\ref{LP01}, we can again shift the integration of \eqref{VsN} from the line ${\rm Re}\lambda=\kappa>0$ to $\mbox{Re}\lambda=-\sigma_0$ to obtain
\bq
 V_{\kappa,N}=V_{-\sigma_0,N}+I_N,   \label{VsN1}
 \eq
with
$$
 I_N=\frac1{2\pi i}\(\int^{-\kappa+ i N}_{-\sigma_0+ i N}-\int^{-\kappa- i N}_{-\sigma_0- i N}\)
     e^{\frac{\lambda t}{\eps^2}}H_{\eps}(\lambda,\xi)f d\lambda.
$$

By Lemma~\ref{LP03}, it holds that
 \be
 \|I_N\|\rightarrow0 \ \ \mbox{as}\ \  N\rightarrow\infty.\label{VsNa}
  \ee
Moreover, by Lemma~\ref{LP03} and a similar argument as Lemma \ref{LP01}, we can obtain
 \be
 \sup_{\eps|\xi|>r_0/2, y\in\R} \|[I-G_{\eps}(-\sigma_0+ i y,\xi)]^{-1}\| \le C.\label{VsNa4}
  \ee
By \eqref{V_2a}, \eqref{VsNa4} and \eqref{VsNa5}, we have for any $f,g\in L^2_{\xi}(\R^3_v)$
 \bma
&|(V_{-\sigma_0,\infty}(t)f,g)|\leq Ce^{-\frac{\sigma_0 t}{\eps^2}}\int^{+\infty}_{-\infty}|(H_{\eps}(\lambda,\xi)f,g)|dy
\nnm\\
&\leq C(\|K\|+\eps^2r_0^{-1})e^{-\frac{\sigma_0 t}{\eps^2}}\int^{+\infty}_{-\infty}\|(\lambda-A_{\eps}(\xi))^{-1}f\|\|(\overline{\lambda}-A_{\eps}(-\xi))^{-1}g\|dy
\nnm\\
&\leq C(\|K\|+\eps^2r_0^{-1})e^{-\frac{\sigma_0 t}{\eps^2}}(\nu_0-\sigma_0)^{-1}\|f\|\|g\|,\quad \lambda=-\sigma_0+ i y.  \label{VsNb1}
 \ema

From \eqref{VsNb1} and the fact $\|f\|^2\leq\|f\|_{\xi}^2\leq(1+4\eps^2r_0^{-2})\|f\|^2$ for $\eps(1+|\xi|)> r_0$, we have
  \bq
  \|V_{-\sigma_0,\infty}(t)\|_{\xi} \leq Ce^{-\frac{\sigma_0 t}{\eps^2}}(\nu_0-\sigma_0)^{-1}. \label{VsN2}
 \eq
Therefore, we conclude from \eqref{V_3b} and \eqref{VsN}--\eqref{VsN2} that
 \bq
   e^{\frac{t}{\eps^2}B_{\eps}(\xi)}f =e^{\frac{t}{\eps^2}A_{\eps}(\xi)}f +V_{-\sigma_0,\infty}(t), \quad \eps(1+|\xi|)> r_0.  \label{high}
 \eq

The combination of \eqref{low} and \eqref{high} gives rise to \eqref{B_0} with $S_1(t,\xi,\eps)f$ and $ S_2(t,\xi,\eps)f$ defined by
 \bmas
 S_1(t,\xi,\eps)f&=e^{\frac{t}{\eps^2}\lambda_0(|\xi|,\eps)} \Big(f,\overline{\psi_0(\xi,\eps)}\Big)_{\xi}\psi_0(\xi,\eps)1_{\{\eps(1+|\xi|)\leq r_0\}}, \\
 S_2(t,\xi,\eps)f&=\(e^{\frac{t}{\eps^2}Q_{\eps}(\xi)} P_1f+U_{-\frac12,\infty}(t)\)1_{\{\eps(1+|\xi|)\leq r_0\}}\\
                   &\quad+\(e^{\frac{t}{\eps^2}A_{\eps}(\xi)}f+V_{-\sigma_0,\infty}(t)\)1_{\{\eps(1+|\xi|)> r_0\}}.
\emas
In particular, $S_2(t,\xi,\eps)f$  satisfies \eqref{S2} in terms of \eqref{decay_1}, \eqref{decay_2} \eqref{UsN4} and  \eqref{VsN2}.
\end{proof}

\medskip
\noindent {\bf Acknowledgements:}
The author would like to thank  Yan Guo for helpful suggestions. The research of the author
was  supported by the National Science Fund for Excellent Young Scholars No. 11922107, the National Natural Science Foundation of China  grants No. 11671100, and Guangxi Natural Science Foundation Nos. 2018GXNSFAA138210 and 2019JJG110010.



\begin{thebibliography}{99}
\setlength{\itemsep}{-4pt}
\renewcommand{\baselinestretch}{1}
\small

\bibitem{DDP-0} A. Arnold, P.A. Markowich and G. Toscani, On large time asymptotics for drift-diffusion-Poisson systems. Transp. Theory Stat. Phys. 29,(2000), 571-581.

\bibitem{Ukai1} C. Bardos and S. Ukai, The classical incompressible Navier-Stokes limit of the Boltzmann
equation, \emph{Math. Models Methods Appl. Sci.}, 1 (1991), 235-257.

\bibitem{DDP-1} P. Biler, Existence and asymptotics of solutions for a parabolic-elliptic system with nonlinear no-flux boundary
conditions, Nonlinear Analysis 19 (1992), 1121-1136.

\bibitem{DDP-2} P. Biler, W. Hebisch and T. Nadzieja,  The Debye system: existence and large time
behavior of solutions. Nonlinear Anal. 23(9) (1994), 1189-1209.

\bibitem{DDP-3} P. Biler and J. Dolbeault, Long time behavior of solutions of Nernst-Planck and Debye-H\"{u}ckel drift-diffusion systems. Ann. Henri Poincar\'{e} 1 (2000), 461-472.

\bibitem{time-1} L. Bonilla, J.-A. Carrillo and J. Soler, Asymptotic behavior of an initial-boundary
value problem for the Vlasov-Poisson-Fokker-Planck system. SIAM J. Appl. Math.
57(5) (1997), 1343-1372.

\bibitem{VPFP-1} F. Bouchut, Existence and uniqueness of a global smooth solution for the Vlasov-Poisson-Fokker-Planck system in three dimensions. J. Funct. Anal. 111(1) (1993), 239-258.

\bibitem{VPFP-4} F. Bouchut, Smoothing effect for the non-linear Vlasov-Poisson-Fokker-Planck system.
J. Differ. Equ. 122(2), 225-238 (1995)

\bibitem{time-2} F. Bouchut and J. Dolbeault, On long asymptotics of the Vlasov-Fokker-Planck equation
and of the Vlasov-Poisson-Fokker-Planck system with coulombic and newtonian
potentials. Differ. Integral Equ. 8 (1995), 487-514.

\bibitem{VPFP-5} J.-A. Carrillo, Global weak solutions for the initial-boundary value problems to the
Vlasov-Poisson-Fokker-Planck system. Math. Methods Appl. Sci. 21 (1998), 907-938.

\bibitem{VPFP-6} J.-A. Carrillo and J. Soler, On the initial value problem for the Vlasov-Poisson-Fokker-Planck system with initial data in $L^p$ spaces. Math. Methods Appl. Sci. 18(10)(1995), 825-839.

\bibitem{time-3} J.-A. Carrillo, J. Soler and J.-L. Vazquez, Asymptotic behaviour and self-similarity
for the three dimensional Vlasov-Poisson-Fokker-Planck system. J. Funct. Anal. 141 (1996), 99-132.










\bibitem{Ellis} R.S. Ellis and M.A. Pinsky, The first and second fluid approximations to the linearized Boltzmann equation. {\it J. Math. pure et appl.}, 54 (1975), 125-156.

\bibitem{DDP-4} W. Fang and K. Ito, On the time-dependent drift-diffusion model for semiconductors. J.
Differ. Equ. 117 (1995), 245-280.

\bibitem{FL-1} N. El Ghani and  N. Masmoudi, Diffusion limit of the Vlasov-Poisson-Fokker-Planck
system. Commun. Math. Sci. 8(2) (2010), 463-479.

\bibitem{FL-2} T. Goudon, Hydrodynamic limit for the Vlasov-Poisson-Fokker-Planck system:
analysis of the two-dimensional case. Math. Models Methods Appl. Sci. 15(5) (2005), 737-752.

\bibitem{FL-3} T. Goudon, J. Nieto, F. Poupaud and J. Soler, Multidimensional high-field limit of the
electro-static Vlasov-Poisson-Fokker-Planck system. J. Differ. Equ. 213(2) (2005), 418-442.

\bibitem{Hwang} H.J. Hwang and J. Jang, On the Vlasov-Poisson-Fokker-Planck equation
near Maxwellian. Discrete Contin. Dyn. Syst. Ser. B 18(3) (2013), 681-691.


\bibitem{DDP-5} A. J\"{u}ngel, Quasi-hydrodynamic Semiconductor Equations. Progress in Nonlinear
Differential Equations and their Applications, Vol. 41. Birkh\"{a}user, Basel, 2001.

\bibitem{DDP-6} A. Krzywicki and T. Nadzieja, A nonstationary problem in the theory of electrolytes, Q. Appl. Math. 50 (1992),
105-107.


\bibitem{Li2} H.-L. Li, T. Yang and M.Y. Zhong, Spectrum analysis for the  Vlasov-Poisson-Boltzmann system. Arch. Rational Mech. Anal. 241 (2021), 311-355.

\bibitem{Li3} H.-L. Li, T. Yang, J.W. Sun and M.Y. Zhong, Large time behavior of solutions to Vlasov-Poisson-Landau (Fokker-Planck) equations
(in Chinese). Sci. Sin. Math., 46 (2016),  981-1004.

\bibitem{Li4} H.-L. Li, T. Yang and M.Y. Zhong, Diffusion Limit of  the Vlasov-Poisson-Boltzmann System. Kinetic and Related Models, 14(2) (2021), 211-255.

\bibitem{Yu} L. Luo and H.-J. Yu, Spectrum analysis of the linear Fokker-Planck equation. Anal. Appl., 15(3) (2017), 313-331.

\bibitem{Markowich} P.A. Markowich, C.A. Ringhofer and C. Schmeiser, Semiconductor Equations, Springer-Verlag, Vienna, 1990. x+248 pp.

\bibitem{Maxwell} J. Maxwell, On stresses in rarefied gases arising from inequalities of temperature.
Phil. Trans. Roy. Soc. Lond. 170, 231-256 (1879) (Appendix)



\bibitem{FL-4} J. Nieto, F. Poupaud and J. Soler, High-field limit for the Vlasov-Poisson-Fokker-Planck system. Arch. Ration. Mech. Anal. 158(1) (2001), 29-59.


\bibitem{Pazy} A. Pazy, \emph{Semigroups of Linear Operators and
Applications to Partial Differential Equations}. AMS Vol. 44. Springer-Verlag, New York, (1983)

\bibitem{FL-5} F. Poupaud and J. Soler, Parabolic limit and stability of the Vlasov-Fokker-Planck
system. Math. Models Methods Appl. Sci. 10(7) (2000), 1027-1045.

\bibitem{VPFP-2} G. Rein and J. Weckler, Generic global classical solutions of the Vlasov-Fokker-Planck-Poisson system in three dimensions. J. Differ. Equ. 99(1) (1992), 59-77.

\bibitem{Ukai} S. Ukai and T. Yang, {\it Mathematical Theory of Boltzmann Equation}. Lecture Notes Series-No. 8, Hong Kong: Liu Bie Ju Center for Mathematical Sciences, City University of Hong Kong, March 2006.


\bibitem{VPFP-7} H.D. Victory, On the existence of global weak solutions for Vlasov-Poisson-Fokker-Planck systems. J. Math. Anal. Appl. 160(2) (1991), 525-555.

\bibitem{VPFP-3} H.D. Victory and B.P. O'Dwyer, On classical solutions of Vlasov-Poisson-Fokker-Planck systems. Indiana Univ. Math. J. 39(1) (1990), 105-156.


\bibitem{FL-6}H. Wu, T.-C. Lin and C. Liu, Diffusion Limit of Kinetic Equations
for Multiple Species Charged Particles, Arch. Rational Mech. Anal. 215 (2015), 419-441.


\end{thebibliography}
\end{document}